\documentclass[10pt,twoside,reqno]{amsart}

\usepackage[
paperwidth=510.236pt,
paperheight=751.181pt,
inner=70pt,
outer=70pt,
top=77pt,
bottom=65pt,
headheight=12pt,
headsep=12pt
]{geometry}

\usepackage[T1]{fontenc}
\usepackage[utf8]{inputenc}
\usepackage{lmodern}
\usepackage{microtype}

\usepackage{amsmath,amssymb,amsfonts,amsthm}
\usepackage{xcolor}
\usepackage{graphicx,enumerate}
\usepackage{textcase}
\usepackage{fancyhdr}
\usepackage[colorlinks=true,linkcolor=black,citecolor=black,urlcolor=black]{hyperref}
\usepackage[nameinlink,noabbrev]{cleveref}

\newcommand{\ShortTitle}{}
\newcommand{\ShortAuthors}{\hfill \textsc{L. Bouthat, N. Doyon, J. Mashreghi, F. Morneau-Guerin}}

\fancypagestyle{firstpage}{%
	\fancyhf{}%
}

\newcommand{\Header}{%
	\thispagestyle{firstpage}%
	\,\par
	\,\par
	\vspace{-10pt}
}

\newcommand{\Title}[1]{%
	\begin{center}
		{\large\bfseries\MakeTextUppercase{#1}\par}
	\end{center}
	\vspace{15pt}
}

\newcommand{\Authors}[1]{%
	\begin{center}
		{\scshape #1\par}
	\end{center}
	\vspace{18pt}
}

\newcommand{\Abstract}[1]{%
	\begingroup
	\small
	\noindent\textbf{Abstract:} #1\par
	\endgroup
	\vspace{15pt}
}

\newcommand{\MSC}[1]{%
	\begingroup
	\small
	\noindent\textbf{2020 Mathematics Subject Classification:} #1.\par
	\endgroup
}

\newcommand{\Keywords}[1]{%
	\begingroup
	\small
	\noindent\textbf{Key words:} #1.\par
	\endgroup
	\vspace{18pt}
}

\numberwithin{equation}{section}

\theoremstyle{plain}
\newtheorem{thm}[equation]{Theorem}
\newtheorem*{Perron}{Perron--Frobenius theorem}
\newtheorem{lem}[equation]{Lemma}

\newtheorem{cor}[equation]{Corollary}

\theoremstyle{definition}

\theoremstyle{remark}

\newcommand{\ds}{\Omega_n}

\begin{document}
	
	\Header
	
	\Title{On the convergence\\of doubly stochastic Markov chains}
	
	\Authors{%
		Ludovick Bouthat, Nicolas Doyon, Javad Mashreghi,\\
		and Frédéric Morneau-Guérin%
	}
	
	\Abstract{%
		We characterize the asymptotic behavior of time-homogeneous doubly stochastic Markov
		chains. Our investigation revolves around understanding the dynamics of products of
		doubly stochastic matrices, which in turn allows us to fully characterize three distinct
		behaviors: cyclicity, convergence towards a special equilibrium matrix, and divergence.
		Notably, we introduce a novel and comprehensive sufficient condition for the convergence
		of an infinite product of doubly stochastic matrices.
	}
	
	\MSC{15B51, 60J10}
	

%%%%%%%%%%%%%%%%%%%%%%%%%%%%%%%%%%%%%%%%%%%%%%%%%%%%%%%
%%%%%%%%%%%%%%%%%%%%%%%%%%%%%%%%%%%%%%%%%%%%%%%%%%%%%%%
%%%%%%%%%%%%%%%%%%%%%%%%%%%%%%%%%%%%%%%%%%%%%%%%%%%%%%%
%%%%%%%%%%%%%%%%%%%%%%%%%%%%%%%%%%%%%%%%%%%%%%%%%%%%%%%

\section{Introduction}

A discrete-time Markov chain is a sequence $X_0, X_1,  X_2,  \dots$, or $(X_n)_{n \geq 0}$, of random variables, the so-called \textit{states}, such that the probability of moving to a future state is independent of the previous states and only depends on the current state; this ``memoryless" property is referred to as the Markov property. The possible values of $X_i$s form an at most countable set $S$ of elements $s_j$, called the state space of the chain.

Markov chains are characterized by their transition matrices at each discrete time step $k$, where the product of these matrices describes the probability of transitioning from one state to another within a given number of time steps. Specifically, if $P_k$ is the transition matrix of the Markov chain at time $k$, the entry $(i,j)$ in the matrix $P_1 P_2 \cdots P_m$ represents the probability that $X_m = s_j$, given that $X_0 = s_i$. In cases where the state space is finite, the transition matrices are stochastic matrices, i.e.,  square matrices with nonnegative real entries and with each row summing to 1. A Markov chain is said to be time-homogeneous, or simply homogeneous, when there exists a stochastic matrix $P$ such that $P_k = P$ for all $k \geq 1$. Otherwise, it is said to be non-homogeneous. For detailed information on Markov chains and stochastic processes, we refer to the excellent book by Levin, Peres and Wilmer \cite{Levin2009}.  For an accessible yet detailed discussion of finite Markov chains (including a study of some of the issues that will be discussed in this article but in a slightly different setting), we refer the reader to \cite{Kemeny}.

Building on this setup, the central object becomes the infinite product of transition matrices, whose convergence has been studied extensively. Indeed, Wolfowitz \cite{Wolfowitz} established in 1963 sufficient conditions for the convergence of products of indecomposable and aperiodic stochastic matrices. His results ensured that, under these conditions, the limiting matrix has identical rows, corresponding to a unique stationary distribution. Building on this, Hajnal \cite{Hajnal} extended in 1976 the analysis to products of general nonnegative matrices, proving that under wide conditions, the product of a large number of square non-negative matrices, independent of order, is close to a positive matrix of rank 1. A year later, Chatterjee and Seneta \cite{Chatterjee} refined these ideas in the context of repeated averaging, proving that under suitable connectivity and aperiodicity assumptions, successive products of stochastic matrices converge to a rank-one stochastic matrix, thereby formalizing the emergence of consensus in distributed systems.

In some applications, we require the transition matrices of Markov chains to be \emph{doubly stochastic}. For example, such matrices appear in the study of gossip algorithms by Boyd et al. \cite{Boyd2006} and in the analysis of non-selective quantum measurements by Vourdas \cite{VOURDAS2022126911}.

Recall that an $n\times n$ matrix is doubly stochastic if all of its entries are nonnegative real numbers and if each row and column sums to 1. Birkhoff’s theorem \cite[Theorem 8.7.2]{HornJohnson2013} provides a geometric characterization of doubly stochastic matrices, asserting that every such matrix can be written as a convex combination of $n \times n$ permutation matrices. The associated Birkhoff--von Neumann algorithm offers a constructive procedure for obtaining this decomposition. In the present paper, we shall denote by $\ds$ the set of all $n \times n$ doubly stochastic matrices. Note that $\ds$ is closed under classical matrix multiplication. Therefore, given a doubly stochastic Markov chain (i.e., a Markov chain with transition matrices $D_k \in \ds$ for all $k$), any matrix given by a  product of the form $D_1 D_2 \cdots D_m$ is again doubly stochastic.

Doubly stochastic Markov chains received considerable interest in the latter half of the 20th century. However, since the literature on doubly stochastic Markov chains is very varied (both in terms of the language used and the specialized jargon employed), it has proven difficult to compile a comprehensive list of known results and to fully trace their origins. Nevertheless, we will endeavor to provide a substantial, if incomplete, overview in the following.

In 1965, V.\! M.\ Maksimov first turned his attention to homogeneous doubly stochastic Markov chains. He considered the convergence of $D^m$ as $m\to\infty$ for $D\in\ds$, and obtained some results relating to the convergence of such limits which depends on the Birkhoff decomposition of $D$ \cite{Maksimov1965} (see \cite{Birkhoff1946} for more detail on the Birkhoff decomposition). In view of this line of research, it does not seem unreasonable to assume that Maksimov as well as some of his contemporaries were aware of the following key result, even though the earliest written mention of it that we have been able to trace is in a 1979 book by Fritz, Huppert and Willems \cite[p.\,30]{Fritz1979}: given $D\in\ds$ with $1$ as a simple eigenvalue and such that $\lim_{m\to\infty} D^m$ exists, then
\begin{align}\label{convergence_simple}
    \lim_{m\to\infty} D^m \,=\, \frac{1}{n}\begin{bmatrix}
        1&1&\!\!\cdots\!\!&1 \\
        1&1&\!\!\cdots\!\!&1 \\[-4pt]
        \vdots&\vdots&\!\!\ddots\!\!&\vdots \\
        1&1&\!\!\cdots\!\!&1
    \end{bmatrix}_{n \times n} \,=:\, J_{n} .
\end{align}

In 1970, Maksimov delved deeper into this issue and considered the case of non-homogeneous doubly stochastic Markov chains \cite{Maksimov1970}. In this context, it is the associated sequence of transition matrices $(D_k)_{k\geq 1}$ that is considered. A sequence of such matrices is said to be \emph{convergent} if, for any $k \geq 1$, the matrix product $D_k D_{k+1} \cdots D_m = B_{k,m}$, converges as $m\to\infty$ to some $B_k$. If all these limit matrices $B_k$ coincide (that is, if $B_k$ is independent of $k$), then we say that the chain $(D_k)_{k\geq 1}$ is \emph{properly convergent}. 

Using this notation, Maksimov provided necessary and sufficient conditions ensuring convergence and proper convergence of the infinite product of a given sequence $(D_k)_{k\geq 1} \subseteq \ds$. However, these general conditions relied heavily on the structure of the zero entries in the doubly stochastic matrices. Thus, these conditions prove most useful when the $D_k$'s are connected by some regularity law. For the general case, Maksimov was able to provide a convenient sufficient conditions for (proper) convergence. In particular, he proved that if $d_{ii}^{(k)} \geq \varepsilon > 0$ hold uniformly (i.e., for all $i$ and $k \geq 1$), and if $\sum_{k=1}^\infty d_{ij}^{(k)} + d_{ji}^{(k)}=\infty$ for all $i$ and $j$, then $(D_k)_{k\geq 1}$ converges properly.

In 1980, Schwarz took interest in the non-homogeneous problem and tried to improve upon the sufficient condition given by Maksimov \cite{Schwarz1980}. Therein, he showed among other things that if a doubly stochastic Markov chain $(D_k)_{k\geq 1}$ satisfies the condition $\sum_{k=1}^\infty \nu(D_k)=\infty$, where $\nu(D):= \min_{i,j} d_{ij}$, then the infinite product $D_1 D_2 D_3 \cdots$ converges (in the sense that $L= \lim\limits_{i \rightarrow \infty} D_1 \cdots D_i$ exists) to the uniform matrix $J_n$.

Then, in 1987, Du \cite{Du1987} provided further insight by showing that if $D$ is an irreducible $n\times n$ doubly stochastic matrix and $\operatorname{tr}(D) > 0$, then the limit $\lim_{m\to \infty} D^m$ exists and is equal to $J_n$. Fast forward a few years, in 2001, Hwang and Pyo addressed a similar problem by characterizing the doubly stochastic matrices $D$ whose powers $D^m$, $m \geq 1$, eventually stop, i.e., $D^m=D^p$, $m \geq p$, for some positive integer $p$ \cite{Hwang2001}. Somewhat more recently, in 2009, Rénier determined the limits of the powers of locally finite infinite doubly stochastic matrices \cite{Renier2009}. Finally, we note that in 2010, Hunter \cite{Hunter2010} studied the stationary distribution, the generalized inverses of Markovian kernels, the mean and variance of the first passage times and the expected times to mixing of Markov chains with doubly stochastic transition matrices, which are some of the more important properties of Markov chains.

Recent results on the geometry of the set of doubly stochastic matrices gave some insight which allowed us to improve upon some of the known sufficient conditions mentioned above \cite{bmm1, bmm2}. It is the main purpose of this paper to present this new result. 

Additionally, we have identified what appears to be two inaccuracies relayed recently in the literature. Indeed, in \cite[Theorem 1]{Hunter2010} it is stated that for any $n\times n$ doubly-stochastic matrix $P$, $\lim_{k\to\infty} P^k = J_n$. However, the identity matrix serves as a direct counter-example to this statement. Examining the argument used in the paper, the Perron--Frobenius theorem is used, so it may be safely assumed that one needs to make the additional assumption that $D$ is irreducible. Nonetheless, even then it is not that simple. Indeed, in \cite[Theorem 5.1]{Renier2009} it is stated that if $A\in\ds$ is irreducible, and not a permutation matrix, then $\lim_{k\to\infty} A^k = J_n$. Upon closer inspection, one may note that this result also appears to require refinement. For example, taking $A=\left[\begin{smallmatrix} 0&J_2\\J_2&0 \end{smallmatrix}\right]$, we encounter a counterexample where $A$ is irreducible but exhibits periodic behavior.  Specifically, it is easy to show that $A^k=A$ for every odd $k>0$ and that $A^k=\left[\begin{smallmatrix} J_2 & 0 \\0 & J_2 \end{smallmatrix}\right]$ for every even $k>0$. Given these subtleties, we also aim to clarify and provide more insight on the sequence $(D^m)_{m\geq 1}$, where $D\in\ds$, as we characterize the limiting behavior of this sequence as $m\to\infty$.

We begin our study in \Cref{sec - prelim} by establishing some preliminary facts and results about doubly stochastic matrices. Then, in \Cref{sec - mult}, we focus on the following question: when does the infinite product $D_1 D_2 D_3 \cdots$ converge to the uniform matrix $J_n$? Therein, we provide a new sufficient condition which improves a result of Schwarz \cite{Schwarz1980}. We then proceed in \Cref{sec - simple} to characterize the behavior of any homogeneous doubly stochastic Markov chain; that is we characterize the behavior of the sequence $(D^m)_{m\geq 1}$ as $m\to\infty$ for an arbitrary doubly stochastic matrix $D$. In particular, we show that $D^m$ converges when $m\to\infty$ if and only if $D$ is permutation-similar to a direct sum of primitive doubly stochastic matrices and we characterize the doubly stochastic matrices which describe a cycle of order $p$.

Let us now introduce some notations that shall be used throughout the text. Given any pair of matrices $A$ and $B$ of size $m \times n$ and $p \times q$ respectively, the \emph{direct sum} of $A$ and $B$ is the block matrix of size $(m+p) \times (n+q)$ defined by:
$$
A \oplus B :=  \begin{bmatrix} A & 0 \\ 0 & B\end{bmatrix}.
$$
Likewise, the \emph{Kronecker product} of $A$ and $B$ is defined as the $pm \times qn$ block matrix 
  $$A \otimes B \,=
    \begin{bmatrix}
    a_{11}B\!\! & \cdots & \!\!a_{1n}B \\[0pt]
    \vdots\!\! & \ddots & \!\!\vdots \\[0pt]
    a_{m1}B\!\! & \cdots & \!\!a_{mn}B
    \end{bmatrix}.
$$

The family of all complex $n \times n$ matrices is denoted by $M_n$. The spectral radius of a matrix $A$ is $\rho(A) := \max |\lambda|$, where the maximum is taken over all eigenvalues of $A$. We denote the number of eigenvalues of modulus $\rho(A)$, multiplicity counted, by $h(A)$, and if there is no confusion, simply by $h$.

%%%%%%%%%%%%%%%%%%%%%%%%%%%%%%%%%%%%%%%%%%%%%%%%%%%%%%%
%%%%%%%%%%%%%%%%%%%%%%%%%%%%%%%%%%%%%%%%%%%%%%%%%%%%%%%
%%%%%%%%%%%%%%%%%%%%%%%%%%%%%%%%%%%%%%%%%%%%%%%%%%%%%%%
%%%%%%%%%%%%%%%%%%%%%%%%%%%%%%%%%%%%%%%%%%%%%%%%%%%%%%%

\section{Preliminary}\label{sec - prelim}

A square matrix $A$ is \emph{reducible} if it can be transformed to a block lower-triangular form by simultaneous row and column permutations, i.e., if there exists a permutation matrix $P$ such that
$$P^{\intercal} A P \,=\, \begin{bmatrix}
A_1 & 0 \\
A_3 & A_2
\end{bmatrix},$$
where $A_1$ and $A_2$ are square matrices. A square matrix which is not reducible is said to be \emph{irreducible}. One of the most important results on irreducible nonnegative matrices is the celebrated Perron--Frobenius theorem \cite{Frobenius1912,Perron1907}.  The theorem comprises numerous statements (for instance, it says that the principal eigenvector is unique up to scaling and has all entries of the same sign). We present here only those relevant to the present discussion. For a comprehensive treatment of the Perron--Frobenius theorem, see \cite[Chapter 8]{HornJohnson2013}.

\begin{Perron}\label{thm - Perron-Frobenius}
Let $A \in M_n$, $n \geq 2$, be irreducible with nonnegative entries. Then the following hold.
\begin{enumerate}[(i)]
\item\label{thm - Perron-Frobenius_1} There exists a  number $\mu > 0$, called the Perron--Frobenius eigenvalue of $A$, such that $\mu$ is a simple eigenvalue of $A$ and any other eigenvalue $\lambda$ of $A$ verifies $|\lambda| \leq \mu$. Thus $\rho(A)=\mu$.
\item\label{thm - Perron-Frobenius_2} The Perron–Frobenius eigenvalue of $A$ fulfils the row-sum inequalities
\begin{equation}\label{eq - frob0}
\min_{1\leq i \leq n} \sum_{j=1}^{n} a_{ij} \leq \mu \leq  \max_{1\leq i \leq n} \sum_{j=1}^{n} a_{ij}.
\end{equation}
\item\label{thm - Perron-Frobenius_3} The matrix $A$ is similar to $e^{2\pi i /h} A$ and the spectrum of $A$ is invariant under multiplication by $e^{2\pi i /h}$.
\item If $h>1$, there exists a permutation matrix $P$ such that
\begin{equation}\label{eq - frob}
            P^{\intercal}AP =
            \begin{bmatrix}
        0&\!\!\!\smash{A_{{1,2}}}&\!\!\!0&\!\!\!\cdots&\!\!0 \\
        0&\!\!\!0&\!\!\!\smash{A_{{2,3}}}&\!\!\!\cdots&\!\!0 \\[-4pt]
        \vdots&\!\!\!\vdots&\!\!\!\vdots&\!\!\!\ddots&\!\!\vdots \\
        0&\!\!\!0&\!\!\!0&\!\!\!\cdots&\!\!\smash{A_{{h-1,h}}} \\
        \smash{A_{{h,1}}}&\!\!\!0&\!\!\!0&\!\!\!\cdots&\!\!0
\end{bmatrix},
\end{equation}
where the non-zero blocks are square matrices.
\item If $h=1$, then the limit $$\lim_{m \to \infty} \frac{A^m}{\mu^m}$$ exists. This limit is called the \emph{Perron projection} of $A$.
\end{enumerate}
\end{Perron}

In the case of doubly stochastic matrices, Perfect and Mirsky showed that reducible doubly stochastic matrices assume a very particular form as a direct sum of smaller irreducible doubly stochastic matrices \cite{Mirsky1965}. To be more precise, if $D \in \ds$ is reducible, then there exist an $n\times n$ permutation matrix $P$ and an integer $r \geq 2$ such that
\begin{equation}\label{thm - irr}
    P^{\intercal} D P = D_1 \oplus D_2 \oplus \cdots \oplus D_r,
\end{equation}
where the $D_j$s are irreducible doubly stochastic matrices. Hence, upon applying the Perron--Frobenius Theorem to each summand, a fundamental result emerges: the eigenvalues of any doubly stochastic matrix $D$ lie within the closed unit disk $\overline{\mathbb{D}}$ and $\rho(D)=1$. In fact, for irreducible doubly stochastic matrices, the parts \eqref{thm - Perron-Frobenius_1} and \eqref{thm - Perron-Frobenius_2} of the Perron--Frobenius Theorem guarantee that $\rho(D)=1$. For reducible matrices, it follows from \eqref{thm - irr} that
\[
\rho(D) = \rho(P^{\intercal}D P) = \rho(D_1 \oplus D_2 \oplus \cdots \oplus D_r) = \max_{1\leq k \leq r} \rho(D_k) = 1.
\]
Therefore, for a doubly stochastic matrices $D$, the quantity $h$ represents the number of unimodular eigenvalues of $D$.

When $D \in \Omega_n$ is irreducible, \eqref{eq - frob} reveals that $D$ is permutation-similar to a matrix with a special form, reminiscent of the $h\times h$ \emph{circular shift matrix}
\[
K_{h} := \begin{bmatrix}
    0&1&0&\!\!\cdots\!\!&0 \\
    0&0&1&\!\!\cdots\!\!&0 \\[-4pt]
    \vdots&\vdots&\vdots&\!\!\ddots\!\!&\vdots \\
    0&0&0&\!\!\cdots\!\!&1 \\
    1&0&0&\!\!\cdots\!\!&0
\end{bmatrix},
\]
with square matrices instead of 1s. In the particular case where $D$ is not only irreducible, but also doubly stochastic, then Marcus, Minc, and Moyls showed that $h$ must divide $n$ and that each nonzero square matrix has the same size (namely $n':=n/h$), which is not necessarily the case when $D$ is not doubly stochastic \cite{Marcus1961}. More specifically, the authors proved that there exists a permutation matrix $P$ such that
\begin{equation}\label{eq - Perron-Frob form}
    P^{\intercal}DP = \begin{bmatrix}
        0&\!\!\!\smash{D_1}&\!\!\!0&\!\!\!\cdots&\!\!0 \\
        0&\!\!\!0&\!\!\!\smash{D_2}&\!\!\!\cdots&\!\!0 \\[-4pt]
        \vdots&\!\!\!\vdots&\!\!\!\vdots&\!\!\!\ddots&\!\!\vdots \\
        0&\!\!\!0&\!\!\!0&\!\!\!\cdots&\!\!\smash{D_{h-1}} \\
        \smash{D_{h}}&\!\!\!0&\!\!\!0&\!\!\!\cdots&\!\!0
    \end{bmatrix},
\end{equation}
where $D_j \in \Omega_{n'}$, $1\leq j \leq h$, and $n'=n/h$. Note that we can also write
\begin{equation}\label{eq - Perron-Frob form-2}
P^{\intercal}DP
= \left( D_1 \oplus \cdots \oplus D_h \right) \left( K_{h} \otimes I_{n'} \right),
\end{equation}
where $K_{h} \otimes I_{n'}$ is the Kronecker product of $K_{h}$ and $I_{n'}$, and the ordering of factors and summands on the right hand side is important.

If $e$ denotes the uniform vector $(1, 1, \dots, 1)^{\intercal}$ and $D$ is any doubly stochastic matrix, then $De = e$. This is a direct manifestation of the fact that 1 is always an eigenvalue of any doubly stochastic matrix. Furthermore, if $D$ is irreducible, then $1$ is the Perron--Frobenius eigenvalue of $D$, and statement \eqref{thm - Perron-Frobenius_1} of the Perron--Frobenius theorem ensures that this eigenvalue is simple. Notably, this fact serves as a characterizing property of irreducibility in doubly stochastic matrices. To see this, suppose that $D\in \ds$ is reducible. In this case, \eqref{thm - irr} guarantees the existence of a permutation matrix $P$ such that $P^{\intercal}DP = D_1 \oplus \cdots \oplus D_r$, where $r \geq 2$ and the $D_j$s are doubly stochastic matrices. Given that the eigenvalues of $D$ encompass all the eigenvalues of $D_1, D_2, \dots, D_r$, and as each of the $D_j$s possesses $1$ as an eigenvalue, it follows that the geometric multiplicity of $1$ as an eigenvalue of $D$ is strictly greater than 1. It thus follows that
\begin{align}\label{thm - car_irr}
D\in\ds \text{ is irreducible} ~\iff~\, \text{1 is a simple eigenvalue of } D.
\end{align}
Note that for an irreducible matrix $D$, it is possible that $h(D)>1$. For example, for the irreducible matrix
\[
D :=
\begin{bmatrix}
0 & 1 \\
1 & 0
\end{bmatrix},
\]
we have $h(D)=2$.

A matrix $A \in M_n$ is said to be \emph{primitive} if it is nonnegative, irreducible, and has a \textit{unique} (simple) eigenvalue of maximal modulus (i.e., $h(A)=1$). In the case of doubly stochastic matrices, this definition can be simplified. Indeed, if a doubly stochastic matrix $D$ is not primitive, then either it is reducible (in which case the geometric multiplicity of $1$ as an eigenvalue is strictly greater than 1), or else $h(D)>1$. In both cases, we end up with $h(D)>1$. Therefore,
\begin{align}\label{thm - car_irr - prim}
D\in\ds \text{ is primitive} ~\iff~\, h(D)=1.
\end{align}

Lastly, we need the following lemma, partially due to Marcus who established the equivalence of \emph{(i)} and \emph{(iv)} in \cite[Lemma 2]{Marcus1957}.

\begin{lem}\label{lem - eigenvalues_0_carac}
Let $D\in\ds$. Then the following are equivalent.
    \begin{enumerate}[(i)]
        \item $D=J_n$.
        \item $0$ is an eigenvalue of geometric multiplicity $n-1$ for $D$.
        \item $0$ is a singular value of geometric multiplicity $n-1$ for $D$.
        \item $\operatorname{rank}(D)=1$.
    \end{enumerate}
\end{lem}

\begin{proof}
\noindent $(i) \Longrightarrow (ii)$: It suffices to note that $0$ is an eigenvalue of $J_n$ associated to the $n-1$ linearly independent eigenvectors
\[
\vec{v} := \big(1, \omega_n^k, \omega_n^{2k},\dots,\omega_n^{(n-1)k}\big)^{\intercal}, \qquad 1\leq k \leq n-1,
\]
where $\omega_n:=e^{-2\pi i/n}$. The other eigenvalue is $1$, which necessarily must have multiplicity one.

\smallskip
\noindent $(ii) \Longrightarrow (iii)$: Suppose that $0$ is an eigenvalue of geometric multiplicity $n-1$ for $D$ and observe that if $Dv=0$ then $D^{\intercal}Dv=0$. Thus $0$ is a singular value of $D$ associated with the vector $v$. It follows that $D$ has $0$ as a singular value of geometric multiplicity $n-1$. But, since $D^{\intercal}D$ is a doubly stochastic matrix, and since $1$ is always an eigenvalue of any doubly stochastic matrices, we find that $D$ has exactly one non-zero singular value.

\smallskip
\noindent $(iii) \Longrightarrow (iv)$: From the assumption, it immediately follows that $\operatorname{rank}(D)=1$.

\smallskip
\noindent $(iv) \Longrightarrow (i)$:     The rows of $D$ span a subspace of dimension one. Hence, they must be equal since if two rows were distinct, the sum of the elements of these two rows would be different, which is impossible since $D$ is (doubly) stochastic. Hence,
\[
D = \begin{bmatrix} d_1 & d_2 & \!\cdots\! & d_n\\[-3pt] \vdots & \vdots & \ddots & \vdots\\ d_1& d_2 & \!\cdots\! & d_n \end{bmatrix}.
\]
Since the elements of each column of $D$ sums to $1$, we also find that $nd_j=1$ for each $1\leq j \leq n$. Therefore, $d_j = 1/n$ for all $1\leq j \leq n$ and it follows that $D=J_n$.
\end{proof}

%%%%%%%%%%%%%%%%%%%%%%%%%%%%%%%%%%%%%%%%%%%%%%%%%%%%%%%
%%%%%%%%%%%%%%%%%%%%%%%%%%%%%%%%%%%%%%%%%%%%%%%%%%%%%%%
%%%%%%%%%%%%%%%%%%%%%%%%%%%%%%%%%%%%%%%%%%%%%%%%%%%%%%%
%%%%%%%%%%%%%%%%%%%%%%%%%%%%%%%%%%%%%%%%%%%%%%%%%%%%%%%

\section{Non-homogeneous Markov chains}\label{sec - mult}

In this section, we study the case of non-homogeneous doubly stochastic Markov chains. If $(D_\ell)_{\ell\geq 1}$ is a given sequence of $n\times n$ doubly stochastic matrices, we seek to determine sufficient conditions to ensure that
\begin{equation}\label{E:lim-ddd}
\lim_{m\to\infty}D_1D_2\cdots D_m=J_n.
\end{equation}
This question was initially studied by Maksimov \cite{Maksimov1965,Maksimov1970} and Schwarz \cite{Schwarz1980}. In particular, the latter author showed that if
\begin{equation}\label{E:cond-schwarz}
\sum_{\ell=1}^{\infty} v(D_\ell) = +\infty,
\end{equation}
where
\begin{equation}\label{E:cond-schwarz-2}
v(D):= \min_{i,j} d_{ij},
\end{equation}
then \eqref{E:lim-ddd} holds. In what follows, we present a more general result to ensure \eqref{E:lim-ddd}.

If we label the eigenvalues of $D\in \ds$ as $1=\lambda_1,\lambda_2,\dots,\lambda_n \in \mathbb{C}$, the eigenvalues $\lambda_i$ for $i\geq 2$ are referred to as the \emph{non-stochastic} eigenvalues of $D$. In precise terms, these are the eigenvalues of $D$ whose associated eigenvectors are not scalar multiples of the uniform vector $e$.

\begin{lem}\label{lem - A-J_v.p.}
    Let $D\in \ds$ with eigenvalues $1=\lambda_1,\lambda_2,\dots,\lambda_n \in \mathbb{C}$, and let $\alpha\in\mathbb{C}$ be a complex number. Then the eigenvalues of $D-\alpha J_n$ are $1-\alpha, \lambda_2, \dots, \lambda_n$.
\end{lem}

\begin{proof}
To establish the claim, first observe that
	\begin{enumerate}[(i)]
		\item the matrices $D$ and $J_n$ commute;
		\item the eigenvalue $\lambda_1=1$ of $D$ is associated with the all-ones eigenvector $e$;
		\item and the eigenvalues of $J_n$ are $1$ with geometric multiplicity one (associated with the eigenvector $e$), and 0 with geometric multiplicity $n-1$.
	\end{enumerate}
	It follows from Theorem 2.3.3 in \cite{HornJohnson2013} that there is a unitary matrix $U$ whose first column is equal to $\tfrac{1}{\sqrt{n}} e$ such that
	\[
	U^* D U = \begin{bmatrix}
		1 & 0 &\!\!\cdots\!\! & 0 \\
		* & \lambda_2 & \!\!\cdots\!\! & 0 \\[-3pt]
		\vdots & \vdots & \!\!\ddots \!\!&\vdots \\
		* & * & \!\!\cdots\!\! & \lambda_n
	\end{bmatrix}
	~\quad\& ~\quad
	U^* J_n U  = \begin{bmatrix}
		1 & 0 &\!\!\cdots\!\! & 0 \\
		* & 0 & \!\!\cdots\!\! & 0 \\[-3pt]
		\vdots & \vdots & \!\!\ddots\!\! &\vdots \\
		* & * & \!\!\cdots\!\! & 0
	\end{bmatrix}.
	\]
	Thus,
	\[
	U^*(D-\alpha J_n) U \,=\, \begin{bmatrix}
		1 & 0 &\!\!\cdots\!\! & 0 \\
		* & \!\lambda_2\! & \!\!\cdots\!\! & 0 \\[-3pt]
		\vdots & \vdots & \!\!\ddots\!\! &\vdots \\
		* & * & \!\!\cdots\!\! & \!\lambda_n\!
	\end{bmatrix} - \begin{bmatrix}
		\alpha & 0 &\!\!\cdots\!\! & 0 \\
		* & 0 & \!\!\cdots\!\! & 0 \\[-3pt]
		\vdots & \vdots & \!\!\ddots\!\! &\vdots \\
		* & * & \!\!\cdots\!\! & 0
	\end{bmatrix} \,=\, \begin{bmatrix}
		1-\alpha\!\! & 0 &\!\!\cdots\!\! & 0 \\
		~* & \!\!\lambda_2\! & \!\!\cdots\!\! & 0 \\[-3pt]
		~\vdots & \vdots & \!\!\ddots\!\! &\vdots \\
		~* & * & \!\!\cdots\!\! & \!\lambda_n \!
	\end{bmatrix}.
	\]
	The conclusion follows from this decomposition.
\end{proof}

We now introduce the main result and subsequently explore its scope and how it relates to previously established findings.

\begin{thm}\label{thm - convergence_general}
Let $(D_\ell)_{\ell\geq 1}$ be a sequence of doubly stochastic matrices in $\Omega_n$. If
\[
\sum_{\ell=1}^\infty \big(1-\sigma_2(D_\ell)\big) = +\infty,
\]
where $\sigma_2(D_\ell)$ denotes the second largest singular value of $D_\ell$, then
\[
\lim_{m\to\infty} D_1D_2\cdots D_m = J_n.
\]
\end{thm}

\begin{proof}
The peculiar identity
\[
D_1 D_2\cdots D_m-J_n = (D_1- J_n)(D_2- J_n)\cdots (D_m- J_n)
\]
is proved by induction. By the submultiplicative property of the spectral norm $\|\cdot\|_2$, we have
\begin{equation}\label{E:submul-1}
\|D_1 \cdots D_m-J_n\|_{2} \leq \|D_1- J_n\|_{2}\cdots \|D_m-J_n\|_{2}.  
\end{equation}
Moreover, Birkhoff's theorem and the permutation invariance of both $J$ and the spectral norm ensures that 
\begin{align}
\|D-J_n\|_{2} &=  \bigg\| \sum_{i} \alpha_i (P_i -J_n) \bigg\|_2 \notag\\
&\leq \sum_{i} \alpha_i \|P_i-J_n\|_2 \label{E:submul-2}\\
&= \sum_{i} \alpha_i \|I_n-J_n\|_2 = 1, \notag
\end{align}
for all $D\in\ds$. Here, we used the easily verified identity $\|I_n-J_n\|_2=1$.

Recall that if a sequence $(a_\ell)_{\ell\geq 1}$ of real numbers satisfy $0 \leq a_\ell \leq 1$ for all $\ell$, then $\sum_{\ell=1}^\infty (1-a_\ell) = +\infty$ implies that $\prod_{\ell=1}^\infty a_\ell = 0$. Indeed, since $0\leq x\leq e^{x-1}$ for all $0\leq x\leq 1$, it follows that
\[
0\leq \prod_{k=1}^m a_k \leq \prod_{k=1}^m e^{a_k-1} = e^{-\sum_{k=1}^m (1-a_k)}.
\]
Since $e^{-x}$ is continuous, letting $m\to \infty$ yields $\prod_{k=1}^\infty a_k = 0$. Hence, by \eqref{E:submul-2}, 
\[
\sum_{\ell=1}^\infty \big(1-\|D_\ell-J_n\|_{2}\big) = +\infty
\quad\Longrightarrow\quad \prod_{\ell=1}^\infty \|D_\ell-J_n\|_{2} = 0.
\]
Then, by \eqref{E:submul-1}, the condition on the left-hand side implies
\[
\lim_{m \to \infty} D_1 \cdots D_m = J_n.
\]

To finish, we find yet another interpretation for $\|D_\ell-J_n\|_{2}$. It is well know that $\|D_\ell-J_n\|_{2}$ equals to the greatest singular value of $D_\ell-J_n$. But,
\[
(D_\ell-J_n)^*(D_\ell-J_n) = (D_\ell^{\intercal}-J_n)(D_\ell-J_n) = D_\ell^{\intercal}D_\ell-J_n.
\]
Thus, the largest eigenvalue of $(D_\ell-J)^*(D_\ell-J_n)$ coincides with the largest eigenvalue of $D_\ell^{\intercal}D_\ell-J_n$, which, by \Cref{lem - A-J_v.p.},  is the largest non-stochastic eigenvalue of $D_\ell^{\intercal}D_\ell$, which in turn is the second largest eigenvalue of $D_\ell^{\intercal}D_\ell$. Hence,
\[
\|D_\ell-J_n\|_{2} = \sigma_2(D_\ell). \tag*{\qedhere}
\]
\end{proof}

Let us show that Theorem \ref{thm - convergence_general} implies the Schwarz's result. To establish this fact, we need to make a short digression and perform a delicate calculation. Given a real number $\alpha$,  we have
\[
\|D-\alpha J_n\|_{2} = \sigma_{\max} (D-\alpha J_n) = \rho^{1/2}(D^{\intercal}D-\alpha(2-\alpha)J_n),
\]
since $(D-\alpha J_n)^{\intercal}(D-\alpha J_n) = D^{\intercal}D-\alpha(2-\alpha)J_n$.
\Cref{lem - A-J_v.p.} then ensures that the eigenvalues of the matrix $D^{\intercal}D-\alpha(2-\alpha)J_n$ are $1-\alpha(2-\alpha) = (1-\alpha)^2$, as well as the $n-1$ non-stochastic eigenvalues of $D^{\intercal}D$ (which are the same as the $n-1$ non-stochastic eigenvalues of $D-\alpha J_n$). Consequently,
$$\|D-\alpha J_n\|_{2} = \max\{ |1-\alpha|, \|D-J_n\|_{2} \},$$
from which it also follows that
\begin{equation}\label{E:estim-1}
\|D- J_n\|_{2} = \min_\alpha \|D-\alpha J_n\|_{2} \leq \|D-n v(D) J_n\|_{2},
\end{equation}
where we recall $v(D)$ which was defined in \eqref{E:cond-schwarz-2}.

On the one hand, for $D\neq J_n$, the matrix $\tfrac{1}{1-nv(D)}(D-n v(D) J_n)$ is doubly stochastic, since each row/column sums to 1 and each coefficient is greater than $\tfrac{1}{1-nv(D)}(v(D)-n v(D) \frac{1}{n}) = 0$. This implies that the spectral norm of this latter matrix is 1, since the spectral norm of every doubly stochastic matrix is $1$. Thus
\[
\|D-n v(D) J_n\|_{2} = (1-nv(D)) \left\| \frac{D-n v(D)J_n}{1-nv(D)} \right\|_{2} = 1-nv(D).
\]
As a result, by \eqref{E:estim-1}, we deduce $\|D- J_n\|_{2} \leq 1-nv(D)$. Thus,
\begin{equation}\label{eq - final2}
n v(D) \leq 1 - \|D- J_n\|_{2} = 1-\sigma_2(D)
\end{equation}
for every doubly stochastic matrix $D\neq J_n$.

On the other hand, one can check that \eqref{eq - final2} is trivially verified if $D=J_n$. Hence this relation holds for any $D\in\ds$ and  we conclude that
\[
\sum_{\ell=1}^\infty v(D_\ell) = +\infty
\quad\Longrightarrow\quad
\sum_{\ell=1}^\infty \big(1-\sigma_2(D_\ell)\big) = +\infty.
\]
Consequently, whenever the infinite product of a sequence ${D_1,D_2,\dots} \in \ds$ verifies Schwarz's sufficient condition \eqref{E:cond-schwarz} for convergence, it also verifies the convergence criterion stated in \Cref{thm - convergence_general}.

In order to show that the latter criterion is not simply a reformulation of \eqref{E:cond-schwarz}, but rather a proper generalization, consider the doubly stochastic matrix
\[
D_\ell = \begin{bmatrix}
	1-\varepsilon_\ell-\varepsilon_\ell^2&\varepsilon_\ell&\varepsilon_\ell^2\\
	\varepsilon_\ell&1-2\varepsilon_\ell&\varepsilon_\ell\\
	\varepsilon_\ell^2&\varepsilon_\ell&1-\varepsilon_\ell-\varepsilon_\ell^2
\end{bmatrix},
\]
where $0 \leq \varepsilon_\ell\leq \frac{1}{2}$ goes to 0 as $\ell\to\infty$. Clearly, $v(D_\ell)=\varepsilon_l^2$. Moreover, an easy computation reveals that the eigenvalues of $D$ are $1, 1-3\varepsilon_\ell$ and $1-\varepsilon_\ell-2\varepsilon_\ell^2$. Since the matrix is positive semidefinite, they are also equal to the singular values. Hence, if $\varepsilon_\ell = \frac{1}{4\ell}$, then Schwarz' criteria yield
\[
\sum_{\ell=1}^\infty v(D_\ell) = \frac{1}{16}\sum_{\ell=1}^\infty \frac{1}{\ell^2} < \infty,
\]
while \Cref{thm - convergence_general} gives
\[
\sum_{\ell=1}^\infty (1-\sigma_2(D_\ell)) = \frac{1}{4}\sum_{\ell=1}^\infty \left(\frac{1}{\ell} + \frac{1}{2\ell^2}\right) = \infty.
\]
Hence, $D_1D_2D_3\cdots=J_n$, and the above discussion establishes the more general scope of \Cref{thm - convergence_general}.

However, grouping the matrices together by adjacent pairs and looking at the smallest coefficient of the resulting matrices might allow one to apply Schwarz' criteria, while being easier to verify than \Cref{thm - convergence_general}. We show that this is not possible, even if we allow to group the matrices in blocks of large sizes (i.e., to look at the smallest entry of $D_1D_2\cdots D_k$, then $D_{k+1}D_{k+2}\cdots D_{2k}$, where $k$ is any fixed integer). Indeed, note that the matrices $D_\ell$ can be diagonalized as
\[
D_\ell = \begin{bmatrix}
	1&1&-1\\[3pt]
	1&-2&0\\[3pt]
	1&1&1
\end{bmatrix} \!\!\begin{bmatrix}
	1&0&0\\[3pt]
	0&\alpha_\ell&0\\[3pt]
	0&0&\beta_\ell
\end{bmatrix}\!\!
\begin{bmatrix}
	\frac{1}{3}&\frac{1}{3}&\frac{1}{3}\\[3pt]
	\frac{1}{6}&-\frac{1}{3}&\frac{1}{6}\\[3pt]
	-\frac{1}{2}&0&\frac{1}{2}
\end{bmatrix},
\]
where $\alpha_\ell := 1-3\varepsilon_\ell$ and $\beta_\ell := 1-\varepsilon_\ell-2\varepsilon_\ell^2$. Hence, if $A_{\ell,k}:=\alpha_\ell \cdots \alpha_{\ell+k-1}$ and $B_{\ell,k}:=\beta_\ell \cdots \beta_{\ell+k-1}$, then
\begin{align*}
	D_\ell \cdots D_{\ell+k-1} &= \begin{bmatrix}
		1&1&-1\\[3pt]
		1&-2&0\\[3pt]
		1&1&1
	\end{bmatrix}\!\!\begin{bmatrix}
		1&0&0\\[3pt]
		0&A_{\ell,k}&0\\[3pt]
		0&0&B_{\ell,k}
	\end{bmatrix}\!\!
	\begin{bmatrix}
		\frac{1}{3}&\frac{1}{3}&\frac{1}{3}\\[3pt]
		\frac{1}{6}&-\frac{1}{3}&\frac{1}{6}\\[3pt]
		-\frac{1}{2}&0&\frac{1}{2}
	\end{bmatrix} \\
	&= J_3 + \frac{1}{6}\begin{bmatrix}
		A_{\ell,k}+3B_{\ell,k}&-2A_{\ell,k}&A_{\ell,k}-3B_{\ell,k}\\
		-2A_{\ell,k}&4A_{\ell,k}&-2A_{\ell,k}\\
		A_{\ell,k}-3B_{\ell,k}&-2A_{\ell,k}&A_{\ell,k}+3B_{\ell,k}.
	\end{bmatrix},
\end{align*}
Since $A_{\ell,k},B_{\ell,k}\geq 0$, the smallest entry of $D_\ell \cdots D_{\ell+k-1}$ is realized either by $\frac{1}{3}-\frac{1}{3}A_{\ell,k}$ or by $\frac{1}{3}+\frac{1}{6}A_{\ell,k}-\frac{1}{2}B_{\ell,k}$. But since $\alpha_\ell \leq \beta_\ell$ for all $\ell$, we have $A_{\ell,k}\leq B_{\ell,k}$, which implies that $A_{\ell,k}-3B_{\ell,k} \leq A_{\ell,k}-3A_{\ell,k} = -2A_{\ell,k}$. Hence, the smallest coefficient of $D_\ell \cdots D_{\ell+k}$ is attained by $\frac{1}{3}+\frac{1}{6}A_{\ell,k}-\frac{1}{2}B_{\ell,k}$. It thus follows that
\begin{align*}
	\sum_{n=1}^\infty v(D_{(n-1)k+1} \cdots D_{nk}) = \sum_{n=1}^\infty \left(\frac{1}{3}+\frac{1}{6}A_{(n-1)k+1,k}-\frac{1}{2}B_{(n-1)k+1,k} \right).
\end{align*}
But now, observe that Stirling's approximation yield
\begin{align*}
	A_{(n-1)k+1,k} &= \alpha_{(n-1)k+1} \cdots \alpha_{nk} = (1-3\varepsilon_{(n-1)k+1})\cdots (1-3\varepsilon_{nk}) \\
	&= \left( 1-\frac{3}{4((n-1)k+1)}\right)\cdots\left( 1-\frac{3}{4nk}\right) \\
	&= \frac{(n-1)k+\frac{1}{4}}{(n-1)k+1}\cdots \frac{nk-\frac{3}{4}}{nk} = \frac{\Gamma(nk+\frac{1}{4})\Gamma((n-1)k+1)}{\Gamma(nk+1)\Gamma((n-1)k+\frac{1}{4})} \\
	%
%	&\asymp% 
%	%
%	\frac{(nk-\frac{3}{4})^{\frac{3}{4}}\cdot((n-1)k)^{\frac{1}{2}}}{(nk)^{\frac{3}{2}}((n-1)k-\frac{3}{4})^{-\frac{1}{4}}} \\
	%
	&\asymp 1-\frac{3}{4n}-\frac{3(1-8k)}{32kn^{2}} + O(n^{-3}).
\end{align*}
Similarly, 
\[
B_{(n-1)k+1,k} \asymp 1-\frac{1}{4n}-\frac{3k+1}{32kn^{2}} + O(n^{-3}).
\]
Therefore, for all positive integer values of $k$, we finally get 
\begin{align*}
	\sum_{n=1}^\infty v(D_{(n-1)k+1} &\cdots D_{nk}) = \sum_{n=1}^\infty \left(\frac{1}{3}+\frac{1}{6}A_{(n-1)k+1,k}-\frac{1}{2}B_{(n-1)k+1,k} \right) \\
	&\asymp \sum_{n=1}^\infty \left(\frac{1}{3}+\frac{1}{6}\left(1-\frac{3}{4n}-\frac{3(1-8k)}{32kn^{2}} + O(n^{-3})\right)\right)\\
	&\qquad\qquad\qquad-\frac{1}{2}\left(1-\frac{1}{4n}-\frac{3k+1}{32kn^{2}} + O(n^{-3})\right)  \\
	&= \sum_{n=1}^\infty \left( \frac{11}{64n^{2}} + O(n^{-3}) \right) \\
	&<\infty.
\end{align*}
Hence, there exists matrices for which the condition of \Cref{thm - convergence_general} is stronger  than Schwarz' criteria, but for which the latter doesn't work no matter how many matrices are grouped together.

Moreover, the condition provided in \Cref{thm - convergence_general} is also convenient if the doubly stochastic Markov chain is not predetermined. Indeed, let
\[
D=\begin{bmatrix} \frac{1}{4}&\frac{1}{2}&\frac{1}{4}\\[4pt]
	\frac{1}{2}&0&\frac{1}{2}\\[4pt]
	\frac{1}{4}&\frac{1}{2}&\frac{1}{4}
\end{bmatrix}.
\]

Then it is easily verified that a direct application of Schwarz' criteria does not allow one to conclude that $D^m \to J_3$ as $m\to\infty$, while \Cref{thm - convergence_general} does. However, applying Schwarz' result to the matrix $D^2$ does allow one to show that the infinite product converges to $J_3$ in less computation than \Cref{thm - convergence_general} would require. Now, suppose that the matrix $D_k$ is the matrix $D$ with a probability of $1/7$, and any of the $3\times 3$ permutation matrices with an equal probability of $1/7$.  Although probabilistic considerations have not appeared so far -- and will play only a modest role here -- they will take on a more significant role in later sections. It is nevertheless natural to introduce this simple stochastic variant now. In this random setting, however, there is no straightforward way to group the matrices $D_k$ together to apply Schwarz' result. However, it is clear that $1-\sigma_2(D_k)$ is equal to $0$ with a probability of $6/7$ and $1/2$ with a probability of $1/7$. Hence, an application of \Cref{thm - convergence_general} allows one to conclude that the infinite product $D_1D_2 \cdots D_m$ converges to $J_3$ almost surely as $m\to\infty$.

	During the reviewing process of this document, the papers \cite{MR96306, Hajnal} were brought to our attention. Among the contributions of these publications, \cite[Theorem 3]{Hajnal} was particularly relevant to the present research. Indeed, for a sequence $(P_k)_{k\geq 1}$ of stochastic matrices, define $U_{r,k}=[u_{ij}^{(r,k)}]$ to be the stochastic matrix
	\begin{equation}
		U_{r,k} := P_{r+1} P_{r+2} \cdots P_{r+k}.
	\end{equation}
	The sequence $(P_k)_{k\geq 1}$ is said to be \emph{weakly ergodic} if for all $i,j,s=1,2,\dots,n$ and all $r\geq0$, we have
	\[
	u_{i,s}^{(r,k)} - u_{j,s}^{(r,k)} \xrightarrow[k\to\infty]{} 0,
	\]
	that is, if all rows of $U_{r,k}$ tend to become the same as $k\to\infty$. Moreover, call any scalar function $\mu(\cdot)$ that is continuous on the set of $n\times n$ stochastic matrices $P$ and satisfies $0 \le \mu(P) \le 1$ a \emph{coefficient of ergodicity}. This coefficient is said to be \emph{proper} if $\mu(P)=1$, or, equivalently, if every rows of $P$ are identical.
	
	Hajnal proved in \cite[Theorem 3]{Hajnal} that if $1-m(\cdot)$ and $\mu(\cdot)$ are two proper coefficients of ergodicity satisfying 
	\begin{equation}\label{eq:2.2}
		m\!\left(P_1P_2\cdots P_k\right)
		\le
		\prod_{i=1}^k \bigl(1-\mu(P_i)\bigr),
	\end{equation}
	for any stochastic matrices $P_i$ $(i=1,\ldots,k)$ and any $k\geq 1$, then $(P_k)_{k\geq 1}$ is weakly ergodic if and only if there exists a strictly increasing subsequence $(i_j)_{j\geq 1}$ of the positive integers satisfying
	\[
	\sum_{j=1}^\infty \mu(U_{i_j,i_{j+1}-i_j}) = \infty.
	\]
	Now, let us make the following three observations:
	\begin{enumerate}
		\item $J_n$ is the only doubly stochastic matrix for which every row is identical.
		
		\item[2.] $\mu(D):=1-\|D-J_n\|_2$ is a proper coefficient of ergodicity on doubly stochastic matrices.
		
		\item[3.] If $m(D)=1-\mu(D)=\|D-J_n\|_2$, then \eqref{eq:2.2} is satisfied (see \eqref{E:submul-1}).
	\end{enumerate}
	
	As consequences, we find that (1) a sequence of doubly stochastic $(D_k)_{k\geq 1}$ is weakly ergodic if and only if for all $r\geq 0$, $\lim_{k\to \infty} D_{r+1}\cdots D_{r+k} = J_n$; and (2) a coefficient of ergodicity $\mu(\cdot)$ is proper if and only if $\mu(D)=1 \iff D=J_n$. Moreover, we find by defining $m(D)=1-\mu(D)=\|D-J_n\|_2$ that Hajnal results reduces, after similar observations to those made in the proof of Theorem \ref{thm - convergence_general}, to the equivalence
	\begin{gather}
		\lim_{k\to \infty} D_{r+1} D_{r+2} \cdots D_{r+k} = J_n, \quad \forall r\geq 0, \label{eq - Hajnal2}\\
		\Updownarrow \nonumber\\
		\exists \text{ increasing sequence } (i_j) \text{ s.t. } \sum_{j=1}^\infty \big(1-\sigma_2(D_{i_{j}+1}\cdots  D_{i_{j+1}}) \big)=\infty. \label{eq - Hajnal}
	\end{gather}
	Moreover, it is clear (simply choose $i_j=j$ for all $j\geq 1$) that the sufficient condition of Theorem \ref{thm - convergence_general} implies \eqref{eq - Hajnal}, and that \eqref{eq - Hajnal2} implies that $D_1D_2 \cdots =J_n$. Hence, Theorem \ref{thm - convergence_general} can be derived from Hajnal's 1976 result.

As a final remark to this section, observe that inequality \eqref{eq - final2} is of independent interest, as it gives us an easy-to-compute upper bound on the size of the second largest eigenvalue of a symmetric doubly stochastic matrix. Such an estimate is useful in various context. For instance, Boyd \textit{et al.} \cite{Boyd2006} showed in 2006 that the averaging time of any \textit{gossip} algorithm depends on the second largest eigenvalue of a symmetric doubly stochastic matrix: the smaller this eigenvalue, the faster the averaging algorithm. These algorithms, somehow based on how an epidemic spreads, aim to distribute a certain amount of information to all members of a group in a decentralized manner (i.e., with no master node aggregating information or distributing tasks). Such algorithms have multiple applications in communication \cite{pmlr-v97-koloskova19a}, in cybersecurity \cite{6855039} and in Big Data \cite{Cao2018GossipBasedLB}.

%%%%%%%%%%%%%%%%%%%%%%%%%%%%%%%%%%%%%%%%%%%%%%%%%%%%%%%
%%%%%%%%%%%%%%%%%%%%%%%%%%%%%%%%%%%%%%%%%%%%%%%%%%%%%%%
%%%%%%%%%%%%%%%%%%%%%%%%%%%%%%%%%%%%%%%%%%%%%%%%%%%%%%%
%%%%%%%%%%%%%%%%%%%%%%%%%%%%%%%%%%%%%%%%%%%%%%%%%%%%%%%

\section{Homogeneous Markov chains}\label{sec - simple}

In this section, we consider the behavior of the sequence $D^m$ as $m\to \infty$, where $D$ is a doubly stochastic matrix. As a matter of fact, in all generality, one can use the Jordan canonical form of $A\in \mathbb{C}^{n\times n}$ to compute the powers of $A$ and the limit of $A^m$ as $m\to\infty$. Using this method, Hensel showed that $\lim_{m\to\infty} A^m$ exists if and only if there exist a nonsingular matrix $P$ such that
\[
P^{-1}AP =
\begin{bmatrix}
  I & 0 \\
  0 & K
\end{bmatrix},
\]
where $\rho(K)<1$ \cite{Hensel1926}. In this case,
\[
\lim_{m\to\infty} A^m = P
\begin{bmatrix}
  I & 0 \\
  0 & 0
\end{bmatrix}
P^{-1}
\]
and in particular, $\lim_{m\to\infty} A^m=0$ if and only if $\rho(A)<1$. However, in many cases, this condition is difficult to verify, and the result does not provide significant insights into the behavior of the sequence $(A^m)_{m\geq 1}$ when $\lim_{m\to\infty} A^m$ does not exist. Therefore, using properties of $\ds$ and tools from matrix theory, we aim to provide a more detailed characterization of the behavior of the sequence $(D^m)_{m\geq 1}$ as $m\to \infty$ for any doubly stochastic matrix $D$. In particular, we shall see that one of three phenomena can occur:
\begin{enumerate}[(i)]
\item The sequence forms a cycle, i.e., $D^{p+1}=D$ for some integer $p\geq1$.
\item The sequence converges to a doubly stochastic matrix of the form
\[
P^{\intercal}\left( J_{n_1}\oplus\cdots\oplus J_{n_r}\right) P,
\]
where $P$ is a permutation matrix.
\item The sequence does not converge and is not cyclic. However, the sequence follows a pattern which can be loosely described as \emph{converging to a cycle}.
\end{enumerate}

To describe more precisely the above cases, we say that a matrix $A$ is \emph{cyclic of order $p$} if there is a positive integer $p$ satisfying $A^{p+1}=A$ and $A^{m+1}\neq A$ for all $1\leq m<p$. Moreover, in the language of matrix theory, the matrix $A$ is called \emph{convergent} if $\lim_{m\to\infty} A^m=0$ and \emph{semi-convergent} whenever $\lim_{m\to\infty} A^m$ exists. Here, we are only concerned with semi-convergence since no doubly stochastic matrix is convergent. Furthermore, in our setting, the concepts of proper convergence and convergence (in the sense of Markov chains), defined in the introduction, coincide with the one of semi-convergence. Hence, we shall adopt the classical terminology of semi-convergence mentioned above to avoid any potential confusion.

The work on the convergence problem of doubly stochastic matrices mainly began around 1965. Since then, several results have been proven. However, in recent years, there seems to be misconceptions around some of these results, and there are a handful of incorrect theorems in the literature. We aim to provide a complete characterization of the doubly stochastic matrices $D$ satisfying $\lim_{k\to\infty} D^k = J_n$ in \Cref{thm - converging} and \Cref{thm - converging_general}.

As the first step, the crucial property \eqref{thm - irr} allows us, without any loss of generality,  to assume that the matrix $D\in\ds$ is irreducible. Indeed, for any permutation matrix $P$ and any doubly stochastic matrix $D$, the behavior of the matrix $D^m$ when $m\to \infty$ is entirely determined by that of $\left(P^{\intercal}DP\right)^m=P^{\intercal}D^m P$. Moreover, if $P^{\intercal} D P = D_1 \oplus D_2 \oplus \cdots \oplus D_r,$ then
\[
\left(P^{\intercal} D P\right)^m = D_1^m \oplus D_2^m \oplus \cdots \oplus D_r^m.
\]
Thus, it is enough to study the behavior of $D_j^m$ for the irreducible doubly stochastic matrices $D_j$, $1\leq j \leq r$, to characterize the behavior of the sequence $(D^m)_{m\geq 1}$.
Hence, using the same argument with \eqref{eq - Perron-Frob form}, we observe that it suffices to first prove our results for irreducible doubly stochastic matrices of the form
\begin{align}
    D = \begin{bmatrix}
            0&\!\!\!\smash{D_1}&\!\!\!0&\!\!\!\cdots&\!\!0 \\
            0&\!\!\!0&\!\!\!\smash{D_2}&\!\!\!\cdots&\!\!0 \\[-3pt]
            \vdots&\!\!\!\vdots&\!\!\!\vdots&\!\!\!\ddots&\!\!\vdots \\
            0&\!\!\!0&\!\!\!0&\!\!\!\cdots&\!\!\smash{D_{h-1}} \\
            \smash{D_{h}}&\!\!\!0&\!\!\!0&\!\!\!\cdots&\!\!0
        \end{bmatrix} = \left( \bigoplus_{i=1}^{h} D_{i}  \right) \!\left( K_{h} \bigotimes I_{n/h} \right),
\end{align}
where $h=h(D)$ and $D_{i} \in \Omega_{n/h}$, $1\leq i \leq h$.

This section will thus be divided into three different subsections. We begin by identifying the general behavior of the powers of irreducible doubly stochastic matrices. As a corollary, we are then able to characterize the irreducible doubly stochastic matrices which (i) exhibit a cycle of order $p$, (ii) semi-converge, and (iii) diverge. Moreover, we also determine the precise way in which the sequence $(D^m)_{m\geq 1}$ diverges. Finally, we combine the results of these two subsections to characterize the behavior of $D^m$ for a general doubly stochastic matrix $D$.

%%%%%%%%%%%%%%%%%%%%%%%%%%%%%%%%%%%%%%%%%%%%%%%%%%%%%%%

\subsection{Irreducible doubly stochastic matrices}

The following lemma allows us to easily compute the powers of the matrix $D=\left( D_1 \oplus \cdots \oplus D_h \right) \left( K_{h} \otimes I_{n'} \right)$ encountered in \eqref{eq - Perron-Frob form-2}. For the sake of notational simplicity, let us denote $[m]$ as the number $m\!\!\mod h$. We observe that the $m$th power of the matrix $D$ has the structure of the $m$th power of the circular shift matrix $K_h$, where the nonzero coefficients are in the position $(i,[i+m])$, $1\leq i \leq h$, except that instead of having 1s in these entries, we have more complex matrices.

\begin{lem}\label{lem - useful}
Let 
\[
D = \left( D_1 \oplus \cdots \oplus D_h \right) \left( K_{h} \otimes I_{n'} \right) \in M_{n}(\mathbb{C}), 
\]
where $h,n'\in \mathbb{N}$, $n=hn'$, and $D_j\in M_{n'}(\mathbb{C})$ for every $1\leq j \leq h$. Then
\[
D^{m} = \Bigg( \bigoplus_{i=1}^{h} \prod_{j=1}^{m} D_{[i+j-1]}  \Bigg) \!\left( K_{h}^m \otimes I_{n'} \right).
\]
\end{lem}

\noindent {\em Remark}: Since we are in a non-commutative setting, the ordering in products is important. In the above formula, the product should be interpreted as the left multiplication, i.e., 
\[
\prod_{j=1}^{m} A_i := A_1 A_2 \cdots A_m,
\]
where $A_j$ are arbitrary $n \times n$ matrices.

\begin{proof}
The proof is by induction. For $m=1$, the result is trivial. Now suppose that the statement holds true for $m-1$, and write $B_i=\prod_{j=1}^{m-1} D_{[i+j-1]}$. Then
    \begin{align*}
        D^m &= D^{m-1}D = \Bigg( \bigoplus_{i=1}^{h} \prod_{j=1}^{m-1} D_{[i+j-1]}  \Bigg) \!\left( K_{h}^{m-1} \otimes I_{n'} \right)\!D \\
        &=
        \begin{bmatrix}
            0&\!\!\!\cdots&\!\!\!0&\!\!\!\smash{B_1}\!&\!\!\!\cdots&\!\!0 \\[-3pt]
            0&\!\!\!\cdots&\!\!\!0&\!\!\!0&\!\!\!\ddots&\!\!\!0  \\[-3pt]
            \vdots&\!\!\!\cdots&\!\!\!\vdots&\!\!\!\vdots&\!\!\!\cdots\!&\!\!\!\smash{B_{h-[m-1]}} \\
            \smash{B_{h+1-[m-1]}}\!&\!\!\!\cdots&\!\!\!0&\!\!\!0&\!\!\!\cdots&\!\!0 \\[-3pt]
            0&\!\!\!\ddots&\!\!\!0&\!\!\!0&\!\!\!\cdots&\!\!0 \\
            0&\!\!\!\cdots&\!\!\!\smash{B_{h}}&\!\!\!0&\!\!\!\cdots&\!\!0
        \end{bmatrix}\!
        \begin{bmatrix}
            0&\!\!\!\smash{D_1}&\!\!\!0&\!\!\!\cdots&\!\!0 \\
            0&\!\!\!0&\!\!\!\smash{D_2}&\!\!\!\cdots&\!\!0 \\[-3pt]
            \vdots&\!\!\!\vdots&\!\!\!\vdots&\!\!\!\ddots&\!\!\vdots \\
            0&\!\!\!0&\!\!\!0&\!\!\!\cdots&\!\!\smash{D_{h-1}} \\
            \smash{D_{h}}&\!\!\!0&\!\!\!0&\!\!\!\cdots&\!\!0
        \end{bmatrix},
    \end{align*}
    where each $B_j$ lie in the position $(j,[j+m-1])$, $1\leq j \leq h$. The block multiplication of these matrices is everywhere zero, except when we multiply $B_j$ by $D_{[j+m-1]}$. Hence, we obtain
    \begin{align*}
	D^m &=
	\begin{bmatrix}
		0&\!\!\!\cdots&\!\!\!0&\!\!\!\!\!\!\smash{B_1 D_{[m]}}&\!\!\!\cdots&\!\!0 \\[-3pt]
		\vdots&\!\!\!\ddots&\!\!\!\vdots&\!\!\!\vdots&\!\!\!\ddots&\!\!\!\vdots  \\
		0&\!\!\!\cdots&\!\!\!0&\!\!\!0&\!\!\!\cdots&\!\!\!\smash{B_{h-[m]} D_{[h-1]}} \\[1pt]
		\smash{B_{h+1-[m]}D_{[h]}}&\!\!\!\cdots&\!\!\!0&\!\!\!0&\!\!\!\cdots&\!\!0 \\[-3pt]
		\vdots&\!\!\!\ddots&\!\!\!\vdots&\!\!\!\vdots&\!\!\!\ddots&\!\!\vdots \\
		0&\!\!\!\cdots&\!\!\!\smash{B_{h} D_{[m-1]}}\!\!\!\!&\!\!\!0&\!\!\!\cdots&\!\!0
	\end{bmatrix}\\
	&= \Bigg( \bigoplus_{i=1}^{h} \prod_{j=1}^{m} D_{[i+j-1]}  \Bigg) \!\left( K_{h}^m \otimes I_{n'} \right). \tag*{\qedhere}
\end{align*}
\end{proof}

In \Cref{lem - A-J_v.p.}, we see that for any doubly stochastic matrix $D$ with eigenvalues $1,\lambda_2,\dots,\lambda_n$, where $1$ is the Perron root of $D$ associated with the all-ones eigenvector, the eigenvalues of $D-J_n$ are precisely $0,\lambda_2,\dots,\lambda_n$. Consequently, we have $h(D)=1$ if and only if $\rho(D-J_n)<1$, which itself is equivalent to the fact that $\lim_{m\to\infty} (D-J_n)^m=0$. Moreover, we also saw that $(D-J_n)^m=D^m-J_n$. Therefore, it follows directly that
\begin{equation}\label{eq - conv-J}
\lim_{m\to\infty} D^m = J_n \quad\iff\quad h(D) =1.
\end{equation}

In his 1965 paper, Maksimov provided a more general result by showing that, under certain conditions, there exists a number $p$ such that $\lim_{m\to\infty} D^{mp+r}$ exists, for any $1\leq r \leq p$, and he precisely determined the value of this limit. However, his result relied on a Birkhoff decomposition of $D$ and on subgroups and quotient groups \cite{Birkhoff1946}. This approach introduced challenging conditions to verify, and, moreover, the eventual limit is difficult to compute. The following result provides a clear picture of the period $p$ as well as an explicit formula for the limit. Recall the decomposition \eqref{eq - Perron-Frob form}, which is used below.

\begin{thm}\label{thm - convergence_general_simple}
Let $D\in\ds$ be an irreducible doubly stochastic matrix and let $P$ be a permutation matrix such that
\[
P^{\intercal} D P = \left( \bigoplus_{i=1}^{h} D_{i}  \right) \!\left( K_{h} \otimes I_{n'} \right),
\]
where $h:=h(D)$ and $D_{i} \in \Omega_{n'}$, $1\leq i \leq h$ and $n'=n/h$. Then, for every $1\leq r \leq h$,
\[
D^{mh+r} \,\xrightarrow{m\to\infty}\, P\left(K_{h}^r \otimes J_{n'} \right) P^{\intercal}.
\]
\end{thm}

\begin{proof}
According to \Cref{lem - useful}, we have
\[
P^{\intercal}D^{mh+r}P = \bigg( \bigoplus_{i=1}^{h} \prod_{j=1}^{mh+r} D_{[i+j-1]}  \bigg) \!\left( K_{h}^{mh+r} \otimes I_{n'} \right).
\]
Notice that $K_h^{mh+r} = (K_h^h)^m K_h^r = I_h^m K_h^r = K_h^r$. Hence, we have
\begin{align*}
    \bigoplus_{i=1}^{h} \prod_{j=1}^{mh+r} D_{[i+j-1]} \,&=\, \bigoplus_{i=1}^{h} \, \bigg(\prod_{k=0}^{m-1} \prod_{j=kh+1}^{(k+1)h} D_{[i+j-1]} \bigg) \bigg(\prod_{j=mh+1}^{mh+r} D_{[i+j-1]} \bigg)\\
    &=\, \bigoplus_{i=1}^{h} \, \bigg(\prod_{k=0}^{m-1} \prod_{j=1}^{h} D_{[i+j+kh-1]} \bigg) \bigg(\prod_{j=1}^{r}  D_{[i+j+mh-1]} \bigg)\\
    &=\, \bigoplus_{i=1}^{h} \, \bigg(\prod_{k=0}^{m-1} \prod_{j=1}^{h} D_{[i+j-1]} \bigg) \bigg(\prod_{j=1}^{r}  D_{[i+j-1]} \bigg)  \\
    &=\, \bigoplus_{i=1}^{h} \, \bigg( \prod_{j=1}^{h} D_{[i+j-1]} \bigg)^{\!\!m}  \bigg(\prod_{j=1}^{r} D_{[i+j-1]}\bigg).
    \end{align*}
For simplicity, let us write
\begin{equation}\label{E:def-Eir}
D_{i,r} := D_{[i]}D_{[i+1]}\cdots D_{[i+r-1]}, \qquad r \geq 1.
\end{equation}
Using this notation, it follows that
\[
P^{\intercal}D^{mh+r}P = \left( \bigoplus_{i=1}^{h} D_{i,h}^m D_{i,r} \right) \!\left( K_{h}^{r} \otimes I_{n'} \right).
\]
Now, observe that
\[
\bigoplus_{i=1}^{h} D_{i,h} = \bigoplus_{i=1}^{h} \prod_{j=1}^{h} D_{[i+j-1]} = P^{\intercal} D^h P.
\]
Hence, the number of eigenvalues of unit modulus of $\bigoplus_{i=1}^{h} D_{i,h}$ and $D^h$ is the same, and since the number of unimodular eigenvalues of $D^k$ is independent of $k\geq 1$, we have
\[
h = h(D) = h(D^h) = h(P^{\intercal} D^h P) = h(\oplus_{i=1}^{h} D_{i,h}) = \sum_{i=1}^{h} h(D_{i,h}) \geq h,
\]
where the inequality follows from the fact that each $D_{i,h}$ is doubly stochastic. Therefore, the inequality must be saturated, i.e.,
\[
h(D_{i,h})=1,\qquad 1\leq i \leq h.
\]
Hence, it follows from \eqref{eq - conv-J} and from the fact that $J_nA=J_n$ for every $A\in\ds$ that
\begin{align*}
    \lim_{m\to\infty} P^{\intercal}D^{mh+r}P \,&= \lim_{m\to\infty} \left( \bigoplus_{i=1}^{h} D_{i,h}^m D_{i,r} \right) \!\left( K_{h}^{r} \otimes I_{n'} \right) \\
    &=  \left( \bigoplus_{i=1}^{h} J_{n'} D_{i,r} \right) \!\left( K_{h}^{r} \otimes I_{n'} \right) \\
    &= \left( \bigoplus_{i=1}^{h} J_{n'} \right) \!\left( K_{h}^{r} \otimes I_{n'} \right)\\
    &=\, K_{h}^{r} \otimes J_{n'}. \tag*{\qedhere}
\end{align*}
\end{proof}

\Cref{thm - convergence_general_simple} tells us that the structure of the matrix $P^{\intercal} D^{m} P$ follows a cyclic pattern, while each individual non-zero submatrices semi-converges to the matrix $J_{n'}$. However, we can say even more. Indeed, we can show that the semi-convergence of each individual element is monotone. Therefore, while $P^{\intercal} D^{m} P$ does not itself converges in general, each submatrix gets closer to $J_{n'}$ in a monotonic way.

\begin{cor}
Let $D\in\ds$ be an irreducible doubly stochastic matrix and let $P$ be a permutation matrix such that
\[
P^{\intercal} D P = \left( \bigoplus_{i=1}^{h} D_{i}  \right) \!\left( K_{h} \otimes I_{n'} \right),
\]
where $h:=h(D)$ and $D_{i} \in \Omega_{n'}$, $1\leq i \leq h$ and $n'=n/h$. Then
\[
P^{\intercal} D^{m} P
=
\left( \bigoplus_{i=1}^{h} D_{i,m}  \right) \!\left( K_{h}^m \otimes I_{n'} \right)
\]
and
\[
\|D_{i,m+1}-J_{n'}\|_2 \leq \|D_{i,m}-J_{n'}\|_2, \qquad 1\leq i \leq h,
\]
where $\|\cdot\|_2$ denotes the spectral norm.
\end{cor}

\begin{proof}
The first assertion is a restatement of \Cref{lem - useful}. Then observe that
\[
D_{i,m+1} = \prod_{j=1}^{m+1} D_{[i+j-1]}  = \Bigg( \prod_{j=1}^{m} D_{[i+j-1]} \Bigg) D_{[i+m]} = D_{i,m} D_{[i+m]}.
\]
Since, since $J_n D=J_n$ for any $D\in\ds$, we have
\[
D_{i,m+1}-J_{n'} = D_{i,m} D_{[i+m]} - J_{n'} = \left(D_{i,m}-J_{n'}\right)D_{[i+m]}.
\]
Additionally, as $D^{\intercal}D\in\ds$ and the spectral radius of any doubly stochastic matrix is equal to $1$, it directly follows that $\|D\|_2 = \sqrt{\rho(D^{\intercal}D)} = 1$ for any doubly stochastic matrix $D$. Therefore,
\begin{align*}
    \|D_{i,m+1}-J_{n'}\|_2 &= \left\|\left(D_{i,m}-J_{n'}\right)D_{[i+m]} \right\|_2 \\
    &\leq \left\|D_{i,m}-J_{n'}\right\|_2 \left\|D_{[i+m]} \right\|_2 \\
    &= \left\|D_{i,m}-J_{n'}\right\|_2,
\end{align*}
which concludes the proof.
\end{proof}

%%%%%%%%%%%%%%%%%%%%%%%%%%%%%%%%%%%%%%%%%%%%%%%%%%%%%%%

\subsection{Cyclic and primitive doubly stochastic matrices}

In the following theorem, we characterize the irreducible doubly stochastic matrices which describe a cycle of order $p$.

\begin{thm}\label{thm - cyclic}
Let $D$ be an $n\times n$ irreducible doubly stochastic matrix. Then $D$ is cyclic  if and only if there exists a permutation matrix $P$ such that
\begin{equation}\label{eq - form_D}
    P^{\intercal}DP = K_{h} \otimes J_{n'}.
\end{equation}
where $h:=h(D)$ and $n'=n/h$. Moreover, in that case, $D$ is cyclic of order $h$.
\end{thm}

\begin{proof}
Suppose that $D^{p+1}=D$. By \eqref{eq - Perron-Frob form}, there is a permutation matrix $P$ such that
\[
P^{\intercal} D P = \left( \bigoplus_{i=1}^{h} D_{i}  \right) \!\left( K_{h} \otimes I_{n'} \right),
\]
where $h=h(D)$, $D_{i} \in \Omega_{n'}$, $1\leq i \leq h$, and $n' = n/h$. Let $\lambda_1, \lambda_2,\dots,\lambda_{n}$ denotes the eigenvalues of $D$ and let $U$ be the unitary matrix in a Schur triangularization of $D$ (see \cite[Theorem 2.3.1]{HornJohnson2013}), i.e., a unitary matrix which satisfy
\[
U^* DU = \begin{bmatrix}
		\lambda_1 & 0 &\!\!\cdots\!\! & 0 \\
		* & \lambda_2 & \!\!\cdots\!\! & 0 \\[-3pt]
		\vdots & \vdots & \!\!\ddots \!\!&\vdots \\
		* & * & \!\!\cdots\!\! & \lambda_n
	\end{bmatrix}.
\]
Then, since $D=D^{p+1}$, we have
\[
\begin{bmatrix}
\lambda_1 & 0 &\!\!\cdots\!\! & 0 \\
* & \lambda_2 & \!\!\cdots\!\! & 0 \\[-3pt]
\vdots & \vdots & \!\!\ddots \!\!&\vdots \\
* & * & \!\!\cdots\!\! & \lambda_n
\end{bmatrix}
=
U^* DU = U^* D^{p+1}U = (U^* DU)^{p+1}
=
\begin{bmatrix}
    \lambda_1^{p+1} & 0 &\!\!\cdots\!\! & 0 \\
    * & \lambda_2^{p+1} & \!\!\cdots\!\! & 0 \\[-3pt]
    \vdots & \vdots & \!\!\ddots \!\!&\vdots \\
    * & * & \!\!\cdots\!\! & \lambda_n^{p+1}
\end{bmatrix}.
\]
Consequently, $\lambda_k=\lambda_k^{p+1}$ for all $k \geq 1$; indeed, if a matrix satisfies a polynomial equation, then each of its eigenvalues must be a root of that polynomial. Thus, the eigenvalues of the matrix $D$ are either of unit modulus, or they are equal to 0. However, by the definition of $h$, we know that $D$ has exactly $h$ eigenvalues of unit modulus. Furthermore, the Perron--Frobenius theorem guarantees that the spectrum of $D$ is invariant under multiplication by $e^{\frac{2\pi i}{h}}$. Therefore, the unimodular eigenvalues of $D$ are precisely $e^{\frac{2\pi k i}{h}}$, $1\leq k \leq h$, along with $n-h$ other eigenvalues, all equal to 0.

In particular, it follows that the eigenvalues of $D^{h}$ are 1s (with geometric multiplicity $h$) and $0$s (with geometric multiplicity $n-h$). However, by \Cref{lem - useful},
    \begin{align}\label{eq - cyclique}
         P^{\intercal}D^{h}P &= \begin{bmatrix}
            D_1 \cdots D_{h} & 0 & \cdots & 0 \\
            0 & D_2 \cdots D_{h} D_1 & \cdots & 0 \\[-4pt]
            \vdots & \vdots & \ddots & 0 \\
            0 & 0 & \cdots & D_{h} D_1 \cdots D_{h-1}
        \end{bmatrix} \\
        &=:\! \begin{bmatrix}
            D_{1,h} & 0 & \cdots & 0 \\
            0 & D_{2,h} & \cdots & 0 \\[-4pt]
            \vdots & \vdots & \ddots & 0 \\
            0 & 0 & \cdots & D_{h,h}
        \end{bmatrix}.\nonumber
    \end{align}
Recall the definition of $D_{i,r}$ from \eqref{E:def-Eir}. Since all the matrices in the main diagonal are $n' \times n'$ doubly stochastic matrices, they must all have at least one eigenvalue equal to $1$. Thus, since $D$ has precisely $h$ unimodular eigenvalues, each doubly stochastic matrix in the main diagonal must contains a unique unimodular eigenvalue (equal to 1) while every other eigenvalue is equal to 0. Hence, it follows from \Cref{lem - eigenvalues_0_carac} that $D_{k,h}=J_{n'}$ for every $1\leq k \leq h$.

Lastly, because of the particular form of $D$, given by \eqref{eq - form_D}, the order $p$ must be of the form $p=mh$, for some integer $m\geq 0$, since otherwise we would find using \Cref{lem - useful} that the nonzero block matrices of $P^{\intercal}DP$ and $P^{\intercal}D^{p} P$ would not coincide, which is a contradiction.

Therefore, on the one hand, we have
\[
P^{\intercal}DP =
\begin{bmatrix}
		0&\!\!\!\smash{D_1}&\!\!\!0&\!\!\!\cdots&\!\!0 \\
		0&\!\!\!0&\!\!\!\smash{D_2}&\!\!\!\cdots&\!\!0 \\[-3pt]
		\vdots&\!\!\!\vdots&\!\!\!\vdots&\!\!\!\ddots&\!\!\vdots \\
		0&\!\!\!0&\!\!\!0&\!\!\!\cdots&\!\!\smash{D_{h-1}} \\
		\smash{D_{h}}&\!\!\!0&\!\!\!0&\!\!\!\cdots&\!\!0
\end{bmatrix}.
\]
On the other hand,
\begin{align*}
P^{\intercal}DP &= P^{\intercal}D^p P =(P^{\intercal}DP)^{1+mh} = P^{\intercal}D P\big(P^{\intercal}D^{h} P\big)^m \\
	&= P^{\intercal} DP ( J_{n'}  \oplus \cdots \oplus J_{n'})^m = P^{\intercal}DP ( J_{n'}  \oplus \cdots \oplus J_{n'}) \\
	&= \begin{bmatrix}
		0&\!\!\!\!\!\!\smash{J_{n'}}&\!\!\!0&\!\!\!\cdots&\!\!0 \\
		0&\!\!\!0&\!\!\!\!\smash{J_{n'}}&\!\!\!\cdots&\!\!0 \\[-3pt]
		\vdots&\!\!\!\vdots&\!\!\!\vdots&\!\!\!\ddots&\!\!\vdots \\
		0&\!\!\!0&\!\!\!0&\!\!\!\cdots&\!\!\!\smash{J_{n'}} \\
		\smash{J_{n'}}\!\!\!&\!\!\!0&\!\!\!0&\!\!\!\cdots&\!\!0
	\end{bmatrix} = K_{h} \otimes J_{n'}.
\end{align*}
The above argument is reversible, and the reverse is even simpler.
\end{proof}

We know from \eqref{eq - conv-J} that if $D\in\ds$, then $\lim_{m\to\infty} D^m = J_n$ if and only if $h(D)=1$. As an easy corollary of \Cref{thm - convergence_general_simple}, we shall complete the picture by showing  that $\lim_{m\to\infty}D^m$ exists if and only if $1$ is the only eigenvalue of unit modulus, i.e., $h(D)=1$. It is worth mentioning that Fritz, Huppert and Willems \cite{Fritz1979} provided a sufficient condition which was very close to be necessary (see \cref{convergence_simple}). But our result given here is, on the one hand, the consequence of a much more general result and, on the other hand, necessary to fully characterize the limiting behavior of the sequence $(D^m)_{m\geq 1}$.

\begin{thm}\label{thm - converging}
Let $D \in \Omega_n$ be irreducible. Then $\lim_{m\to\infty}D^m$ exists if and only if $D$ is primitive. Moreover, if so, we necessarily have $\lim_{m\to\infty}D^m = J_n$.
\end{thm}
\begin{proof}
We have the representation  \cref{eq - Perron-Frob form-2}. Moreover, by \Cref{thm - convergence_general_simple}, for every $1\leq r \leq h$, we have
    \[
    D^{mh+r} ~\xrightarrow{m\to\infty}~ P\left(K_{h}^r \otimes J_{n'} \right) P^{\intercal}.
    \]
    Therefore, since $h\geq 1$, a necessary condition for the limit of the powers of $D$ to exist is that $h=1$ (i.e., that $D$ is primitive), otherwise different values of $r$ would yield distinct subsequences with unequal limits. Moreover, if $h=1$, then $r=1$ and \Cref{thm - convergence_general_simple} once again give
    \[
    \lim_{m\to\infty} D^m = \lim_{m\to\infty} D^{m+1} = P\left(K_{1} \otimes J_{n} \right) P^{\intercal} = J_n. \tag*{\qedhere}
    \]
\end{proof}

\begin{cor}
If $m \geq 1$ and $X\in\Omega_n$, then
\[
X^m=J_n
\]
if and only if  $X=J_n+N\geq 0$ (entrywise), where $N$ is a matrix whose rows and columns sum to 0 and for which $N^m=0$.
\end{cor}

\begin{proof}
	Since $X \in \Omega_n$, we have
	\[
	XJ_n = J_nX = J_n .
	\]
	Set $N := X - J_n$. Then
	\[
	N e=0, \qquad e^{\intercal}N=0, \qquad J_nN = NJ_n = 0.
	\]
	Hence, by induction,
	\[
	X^m = (J_n+N)^m = J_n + N^m .
	\]
	Consequently,
	\[
	X^m = J_n \quad \Longleftrightarrow \quad N^m = 0.
	\]
	Finally, the condition $J_n+N \geq 0$ is exactly the nonnegativity condition on $X$.
\end{proof}

%%%%%%%%%%%%%%%%%%%%%%%%%%%%%%%%%%%%%%%%%%%%%%%%%%%%%%%

\subsection{The general case}

In this last part we consider a general doubly stochastic matrix $D$, which is not necessarily irreducible. In this case, the behavior of $D^m$ as $m\to\infty$ depends on the asymptotic behavior of its irreducible submatrices when it is expressed as the direct sum $P^{\intercal}DP = D_1\oplus\cdots\oplus D_r$. Therefore, it is possible for some components to converge while other components may diverge or be cyclic. We use the previously established results to characterize the general structure of doubly stochastic matrices for which the powers converge or exhibit a cycle. Moreover, we also show that the powers of a uniformly chosen doubly stochastic matrix will \emph{almost always} converge to the uniform matrix $J_n$. We begin by characterizing the doubly stochastic matrices which describes a cycle of order $p$.

\begin{thm}
Let $D \in \Omega_n$. Then $D$ is cyclic of order $p$ if and only if there exists a permutation matrix $P$ and positive integers $r, h_i,k_i$, $1\leq i \leq r$, satisfying $p=\operatorname{LCM}(h_1,h_2,\dots,h_r)$ and $h_1k_1 + h_2k_2 + \cdots + h_r k_r =n$ such that
    \[
    P^{\intercal}DP = \bigoplus_{j=1}^r \left( K_{h_j} \otimes J_{k_j} \right).
    \]
    In that case, $D$ has $r$ eigenvalues equal to $1$ and $h(D)=h_1+h_2+\cdots+h_r$.
\end{thm}

\begin{proof}
According to the representation \cref{thm - irr}, there is a permutation matrix $P$ such that
\[
P^{\intercal} D P = D_1 \oplus D_2 \oplus \cdots \oplus D_r,
\]
where each $D_i\in\Omega_{n_i}$ is irreducible with $n_1+\cdots+n_r=n$. Now, $D$ is cyclic of order $p$ if and only if $P^{\intercal} D P$ is cyclic of order $p$. The latter happens if and only if $D_i^{p+1}=D_i$ for all $1\leq i \leq r$ and, moreover, for each integer $m$ with $1\leq m<p$, there exist at least one index $1\leq k\leq r$ such that $D_k^{m+1}\neq D_k$. In particular, every submatrix $D_i$ must be cyclic of order $p_i | p$. Hence, \Cref{thm - cyclic} guarantees that for each $D_i$, there exists a permutation matrix $P$ such that
    \[
    P_i^{\intercal}D_iP_i = K_{p_i} \otimes J_{n_i/p_i} .
    \]
    Therefore, writing $h_i:=p_i$ and $k_i:= n_i/p_i$, we find that there exists a permutation matrix $Q$ such that
    \[
    Q^{\intercal}DQ = \bigoplus_{j=1}^r \left( K_{h_j} \otimes J_{k_j} \right).
    \]
Moreover, each submatrix is cyclic of order $h_j$ and thus, the smallest positive integer $p$ such that $D^{p+1}=D$ is $p=\operatorname{LCM}(h_1,h_2,\dots,h_r)$.
\end{proof}

We now focus on characterizing the doubly stochastic matrices for which the powers semi-converge to some doubly stochastic matrix. As in the previous case, we do not assume that the matrix is irreducible. Recall \eqref{thm - irr} which is used below.

\begin{thm}\label{thm - converging_general}
Let $D \in \Omega_n$, and let $h=h(D)$. Let $P$ be a permutation matrix such that
    \[
    P^{\intercal}DP = D_1 \oplus D_2 \oplus \cdots \oplus D_r,
    \]
    where each $D_i \in \Omega_{n_i}$ is irreducible. Then $D^m$ semi-converges as $m\to\infty$ if and only if $D$ has precisely $h$ eigenvalues equal to $1$. Moreover, in this case,  $r=h$ and $\lim_{m\to\infty} D^m = P (J_{n_1} \oplus J_{n_2} \oplus \cdots \oplus J_{n_r}) P^{\intercal}$.
\end{thm}
\begin{proof}
    The proof is a direct consequence of \eqref{thm - irr} and \Cref{thm - converging}.
\end{proof}

%%%%%%%%%%%%%%%%%%%%%%%%%%%%%%%%%%%%%%%%%%%%%%%%%%%%%%%
%%%%%%%%%%%%%%%%%%%%%%%%%%%%%%%%%%%%%%%%%%%%%%%%%%%%%%%
%%%%%%%%%%%%%%%%%%%%%%%%%%%%%%%%%%%%%%%%%%%%%%%%%%%%%%%
%%%%%%%%%%%%%%%%%%%%%%%%%%%%%%%%%%%%%%%%%%%%%%%%%%%%%%%

\section{A probabilistic result}

When randomly selecting a doubly stochastic matrix of order $n$, what level of certainty can we establish regarding the convergence of its powers? In this regard and with respect to Lebesgue measure, we can show that the powers of \emph{almost all} doubly stochastic matrices will converge to the matrix $J_n$. To verify this claim, recall that if the coefficients of an $n\times n$ matrix $A$ are nonnegative, then $A$ is primitive if and only if the coefficients of $A^m$ are positive for some $m \geq 1$ \cite[Theorem 8.5.2]{HornJohnson2013}. Hence, if the coefficients of $A\in \ds$ are positive, then $A$ is primitive and \Cref{thm - converging} ensures that $A^m \to J_n$ as $m\to \infty$. Consequently, it is enough to convince ourselves that \emph{almost all} doubly stochastic matrix are positive. Intuitively, it suffices to observe that the set $\ds^{+}$ of all-component positive doubly stochastic matrices is precisely the interior of the polytope $\ds$. Hence, since $\ds^{+}$ is nonempty, we should expect a uniformly chosen random doubly stochastic matrix to be in $\ds^{+}$.

\begin{thm}
For almost all doubly stochastic matrices $D$, $D^m$ converges to the matrix $J_n$ as $m\to\infty$.
\end{thm}

\begin{proof}
As discussed before the theorem, it suffice to show that almost all doubly stochastic matrices are in $\Omega_n^+$. We define $\mathcal{E}_{n-1}$, $n \geq 2$, the subset of all $(n-1)\times (n-1)$ matrices as follows.
\begin{enumerate}[(i)]
\item The matrices in $\mathcal{E}_{n-1}$ are nonnegative (componentwise).
\item The row sums and column sums never exceed $1$.
\item The sum of all entries is larger or equal to $n-2$.
\end{enumerate}
In other words, $\mathcal{E}_{n-1}$ is the $(n-1)^2$-dimensional convex polytope in $\mathbb{R}^{(n-1)\times (n-1)}$ consisting of the $(n-1)\times (n-1)$ matrices $A=(a_{i,j})$ in the intersection of the closed half-planes
    \begin{gather}
        a_{i,j} \geq 0, \qquad (1\leq i,j \leq n-1), \label{eq1}\\
        \sum_{j=1}^{n-1} a_{i,j} \leq 1, \qquad (1\leq i \leq n-1), \label{eq2}\\
        \sum_{i=1}^{n-1} a_{i,j} \leq 1, \qquad (1\leq j \leq n-1), \label{eq3}\\
        \sum_{i,j=1}^{n-1} a_{i,j} \geq n-2. \label{eq4}
    \end{gather}

It is easily verified that the mapping $\Phi : \ds \to \mathcal{E}_{n-1}$, given by
    \[
    \begin{bmatrix}
        a_{1,1} & \cdots & a_{1,n-1} & a_{1,n} \\
        \vdots & \ddots & \vdots & \vdots \\
        a_{n-1,1} & \cdots & a_{n-1,n-1} & a_{n-1,n} \\
        a_{n,1} & \cdots & a_{n,n-1} & a_{n,n}
    \end{bmatrix}
    \longmapsto
    \begin{bmatrix}
    	a_{1,1} & \!\!\cdots\!\! & a_{1,n-1}  \\
    	\vdots & \!\!\ddots\!\! & \vdots  \\
    	a_{n-1,1} & \!\!\cdots\!\! & a_{n-1,n-1}
    \end{bmatrix},
    \]
    along with its inverse $\Phi^{-1} :  \mathcal{E}_{n-1} \to \ds$,
    \[
    \begin{bmatrix}
        a_{1,1} & \!\!\cdots\!\! & a_{1,n-1}  \\
        \vdots & \!\!\ddots\!\! & \vdots  \\
        a_{n-1,1} & \!\!\cdots\!\! & a_{n-1,n-1}
    \end{bmatrix}
    \longmapsto
    \begin{bmatrix}
        a_{1,1} & \!\cdots\! & a_{1,n-1} & 1-\sum a_{1,j} \\
        \vdots & \!\ddots\! & \vdots & \vdots \\
        a_{n-1,1} & \!\cdots\! & a_{n-1,n-1} & 1-\sum a_{n-1,j} \\
        1-\sum a_{i,1} & \!\cdots\! & 1-\sum a_{i,n-1} & 2-n + \sum a_{i,j}
    \end{bmatrix},
    \]
    is one-to-one.

Let $\ds^{0}$  denote the set of $n\times n$ doubly stochastic matrices with at least one null coefficient. At the same token, we consider $\ds^{0}(i,j)$,  the set of $n\times n$ doubly stochastic matrices with the coefficient in position $(i,j)$ being zero. Then put
\[
\mathcal{E}_{n-1}^0 := \Phi(\ds^{0}),
\quad\mbox{and}\quad
\mathcal{E}_{n-1}^0(i,j) := \Phi\big(\ds^{0}(i,j)\big).
\]
Now the following relations hold.
\begin{enumerate}[(i)]
        \item If $1\leq i,j \leq n-1$, then $\mathcal{E}_{n-1}^0(i,j)$ correspond to the polytope $\mathcal{E}_{n-1}^0$ with an equality in \eqref{eq1} corresponding to coefficients $(i,j)$.
        \item If $j=n$ and $i<n$, then $\mathcal{E}_{n-1}^0(i,j)$ correspond to the polytope $\mathcal{E}_{n-1}^0$ with an equality in \eqref{eq2} corresponding to coefficient $i$.
        \item If $i=n$ and $j<n$, then $\mathcal{E}_{n-1}^0(i,j)$ correspond to the polytope $\mathcal{E}_{n-1}^0$ with an equality in \eqref{eq3} corresponding to coefficient $j$.
        \item If $i=j=n$, then $\mathcal{E}_{n-1}^0(i,j)$ correspond to the polytope $\mathcal{E}_{n-1}^0$ with an equality in \eqref{eq4}.
\end{enumerate}

In each case, we lose a degree of freedom in the parameters of $\mathcal{E}_{n-1}^0$. Consequently, we have $\dim\big( \mathcal{E}_{n-1}^0(i,j) \big) = (n-1)^2-1 < (n-1)^2$. Therefore, if $m$ denotes the Lebesgue measure on $\mathbb{R}^{(n-1)^2}$, we trivially have $m\big(\mathcal{E}_{n-1}^0(i,j)\big)=0$ for all $1\leq i,j \leq n-1$. It thus follows that
    \[
    m\big(\mathcal{E}_{n-1}^0 \big) = m\left( \cup_{i,j} \mathcal{E}_{n-1}^0(i,j) \right) \leq \sum_{i,j} m\big( \mathcal{E}_{n-1}^0(i,j)\big) = 0.
    \]
    Moreover, observe that
\[
\left[ \tfrac{n-2}{(n-1)^2},\tfrac{1}{n}\right]^{(n-1)^2} \!\!\!\subseteq \mathcal{E}_{n-1},
\]
since for any element $A$ in the former set, $\Phi^{-1}(A) \in \Omega_n$. Consequently,
    \[
    m\big( \mathcal{E}_{n-1} \big) \geq m\left(\left[ \tfrac{n-2}{(n-1)^2},\tfrac{1}{n}\right]^{(n-1)^2}\right) = \left(n(n-1)^2\right)^{-(n-1)^2} > 0.
    \]
    Hence, almost all element of $\mathcal{E}_{n-1}$ belong to $\mathcal{E}_{n-1} \setminus \mathcal{E}_{n-1}^{0}$, the range of $\Phi$ of the positive doubly stochastic matrices. Since $\Phi$ satisfies the conditions of Theorem 31 in \cite{MR0281862}, the conclusion follows.
\end{proof}

\noindent {\em Remark}:  A $n \times n$ doubly stochastic matrix is completely determined by its $(n-1)\times (n-1)$ upper left entries.  There is thus a bijection between the set of $n\times n$ doubly stochastic matrices and the set $\mathcal{V}$ of vectors of $v\in (\mathbb{R}^{+})^{(n-1)^2}$ satisfying the conditions
\begin{enumerate}
\item 
$
\displaystyle \sum_{j=1}^{n-1}v_{m(n-1)+j}\le 1,~~~~~ \forall m \text{ such that } 0\le m\le n-2, 
$
\item 
$
\displaystyle \sum_{j=0}^{n-2} v_{j(n-1)+m}\le 1,~~~~~ \forall m  \text{ such that } 1\le m\le n-1,
$
\item 
$
\displaystyle \sum_{j=1}^{(n-1)^2} v_{j}\ge n-2.
$
\end{enumerate}
This bijection between doubly stochastic matrices and the convex set $\mathcal{V}$ makes it possible to sample matrices using the Lebesgue measure in $\mathbb{R}^{(n-1)^2}$.

From a practical point of view, random doubly stochastic matrices are usually generated by the Sinkhorn process, which consist of taking a random nonnegative matrix and iteratively normalizing the rows and column. This process always converges linearly to a doubly stochastic matrix, as long as the initial nonnegative matrix does not have \emph{too many zeros} (the formal condition is that the matrix has total
support; for more information, we refer to \cite{Sinkhorn}). The important observation here is that the zero coefficients in the doubly stochastic matrix obtained in the limit are precisely in the same positions as the initial nonnegative matrix. Consequently, as long as the initial random nonnegative matrices are obtained in such a way that almost all of them are positive, then almost all doubly stochastic matrices resulting from the Sinkhorn process will be positive and thus will converge to the matrix $J_n$.

%%%%%%%%%%%%%%%%%%%%%%%%%%%%%%%%%%%%%%%%%%%%%%%%%%%%%%%
%%%%%%%%%%%%%%%%%%%%%%%%%%%%%%%%%%%%%%%%%%%%%%%%%%%%%%%
%%%%%%%%%%%%%%%%%%%%%%%%%%%%%%%%%%%%%%%%%%%%%%%%%%%%%%%
%%%%%%%%%%%%%%%%%%%%%%%%%%%%%%%%%%%%%%%%%%%%%%%%%%%%%%%

\section*{Acknowledgement}

We sincerely thank the two anonymous referees for their insightful comments, which have significantly improved the quality of our article. We are also grateful to them for drawing our attention to the works of Hajnal \cite{MR96306, Hajnal} -- classic, to be sure, but presented in a different context and therefore not previously on our radar -- as well as to the related contributions of Wolfowitz \cite{Wolfowitz} and Chatterjee \cite{Chatterjee}. Their careful reading and thoughtful suggestions have been especially valuable.

\bibliographystyle{plain}
\bibliography{ref}

@article {MR96306,
    AUTHOR = {Hajnal, J.},
     TITLE = {Weak ergodicity in non-homogeneous {M}arkov chains},
   JOURNAL = {Proc. Cambridge Philos. Soc.},
  FJOURNAL = {Proceedings of the Cambridge Philosophical Society},
    VOLUME = {54},
      YEAR = {1958},
     PAGES = {233--246},
      ISSN = {0008-1981},
   MRCLASS = {60.00},
  MRNUMBER = {96306},
MRREVIEWER = {K.\ L.\ Chung},
}

@article {Boyd2006,
    AUTHOR = {Boyd, Stephen and Ghosh, Arpita and Prabhakar, Balaji and
              Shah, Devavrat},
     TITLE = {Randomized gossip algorithms},
   JOURNAL = {IEEE Trans. Inform. Theory},
  FJOURNAL = {Institute of Electrical and Electronics Engineers.
              Transactions on Information Theory},
    VOLUME = {52},
      YEAR = {2006},
    NUMBER = {6},
     PAGES = {2508--2530},
   MRCLASS = {94A05 (90B18)},
  MRNUMBER = {2238556},
}

@article{VOURDAS2022126911,
title = {Markov chains with doubly stochastic transition matrices and application to a sequence of non-selective quantum measurements},
journal = {Physica A: Statistical Mechanics and its Applications},
volume = {593},
pages = {126911},
year = {2022},
issn = {0378-4371},
doi = {⁦https://doi.org/10.1016/j.physa.2022.126911⁩},
url = {⁦https://www.sciencedirect.com/science/article/pii/S0378437122000358⁩},
author = {A. Vourdas}
}

@article {Wolfowitz,
    AUTHOR = {Wolfowitz, J.},
     TITLE = {Products of indecomposable, aperiodic, stochastic matrices},
   JOURNAL = {Proc. Amer. Math. Soc.},
  FJOURNAL = {Proceedings of the American Mathematical Society},
    VOLUME = {14},
      YEAR = {1963},
     PAGES = {733--737},
      ISSN = {0002-9939,1088-6826},
   MRCLASS = {15.65},
  MRNUMBER = {154756},
MRREVIEWER = {J.\ Kiefer},
       DOI = {10.2307/2034984},
       URL = {https://doi.org/10.2307/2034984},
}

@article {Chatterjee,
    AUTHOR = {Chatterjee, S. and Seneta, E.},
     TITLE = {Towards consensus: some convergence theorems on repeated
              averaging},
   JOURNAL = {J. Appl. Probability},
  FJOURNAL = {Journal of Applied Probability},
    VOLUME = {14},
      YEAR = {1977},
    NUMBER = {1},
     PAGES = {89--97},
      ISSN = {0021-9002,1475-6072},
   MRCLASS = {60J10},
  MRNUMBER = {428454},
MRREVIEWER = {Richard\ Madsen},
       DOI = {10.2307/3213262},
       URL = {https://doi.org/10.2307/3213262},
}

@article {Hajnal,
    AUTHOR = {Hajnal, J.},
     TITLE = {On products of non-negative matrices},
   JOURNAL = {Math. Proc. Cambridge Philos. Soc.},
  FJOURNAL = {Mathematical Proceedings of the Cambridge Philosophical
              Society},
    VOLUME = {79},
      YEAR = {1976},
    NUMBER = {3},
     PAGES = {521--530},
      ISSN = {0305-0041,1469-8064},
   MRCLASS = {15A48 (60J10)},
  MRNUMBER = {396628},
MRREVIEWER = {E.\ Seneta},
       DOI = {10.1017/S030500410005252X},
       URL = {https://doi.org/10.1017/S030500410005252X},
}

@book {Kemeny,
    AUTHOR = {Kemeny, John G. and Snell, J. Laurie},
     TITLE = {Finite {M}arkov chains},
    SERIES = {Undergraduate Texts in Mathematics},
      NOTE = {Reprinting of the 1960 original},
 PUBLISHER = {Springer-Verlag, New York-Heidelberg},
      YEAR = {1976},
     PAGES = {ix+210},
   MRCLASS = {60J10 (60J20)},
  MRNUMBER = {410929},
}

@article {bmm1,
AUTHOR = {Bouthat,Ludovick and Mashreghi, Javad and Morneau-Guérin, Frédéric},
     TITLE = {On the {G}eometry of the {B}irkhoff {P}olytope {I}. {T}he operator $\ell_n^p$-norms},
   JOURNAL = {Acta Sci. Math. (Szeged)},
  FJOURNAL = {Acta Sci. Math. (Szeged)},
      YEAR = {2024},
      Note = {https://doi.org/10.1007/s44146-024-00152-8}
}

@article {bmm2,
AUTHOR = {Bouthat,Ludovick and Mashreghi, Javad and Morneau-Guérin, Frédéric},
     TITLE = {On the {G}eometry of the {B}irkhoff {P}olytope {I}{I}. {T}he {S}chatten $p$-norms},
   JOURNAL = {Acta Sci. Math. (Szeged)},
  FJOURNAL = {Acta Sci. Math. (Szeged)},
      YEAR = {2024},
      NOTE = {https://doi.org/10.1007/s44146-024-00153-7},
}

@book {MR0281862,
    AUTHOR = {Rogers, C. A.},
     TITLE = {Hausdorff measures},
 PUBLISHER = {Cambridge University Press, London-New York},
      YEAR = {1970},
     PAGES = {viii+179},
   MRCLASS = {28.13},
  MRNUMBER = {281862},
MRREVIEWER = {Edwin\ Hewitt},
}

@article {Birkhoff1946,
    AUTHOR = {Birkhoff, Garrett},
     TITLE = {Tres observaciones sobre el algebra lineal},
   JOURNAL = {Univ. Nac. Tucum\'{a}n. Revista A.},
    VOLUME = {5},
      YEAR = {1946},
     PAGES = {147--151},
   MRCLASS = {09.1X},
  MRNUMBER = {0020547},
MRREVIEWER = {J. L. Dorroh},
}

@article {Schwarz1980,
	AUTHOR = {Schwarz, \v{S}tefan},
	TITLE = {Infinite products of doubly stochastic matrices},
	JOURNAL = {Acta Math. Univ. Comenian.},
	FJOURNAL = {Acta Mathematica Universitatis Comenianae},
	VOLUME = {39},
	YEAR = {1980},
	PAGES = {131--150},
	ISSN = {0862-9544},
	MRCLASS = {15A51 (60J10)},
	MRNUMBER = {619269},
	MRREVIEWER = {Ray C. Shiflett},
}

@article {Marcus1957,
    AUTHOR = {Marcus, Marvin},
     TITLE = {On subdeterminants of doubly stochastic matrices},
   JOURNAL = {Illinois J. Math.},
  FJOURNAL = {Illinois Journal of Mathematics},
    VOLUME = {1},
      YEAR = {1957},
     PAGES = {583--590},
      ISSN = {0019-2082},
   MRCLASS = {15.00},
  MRNUMBER = {95855},
MRREVIEWER = {J. H. Williamson},
       URL = {http://projecteuclid.org.acces.bibl.ulaval.ca/euclid.ijm/1255380681},
}

@book {HornJohnson2013,
    AUTHOR = {Horn, Roger A. and Johnson, Charles R.},
     TITLE = {Matrix analysis},
   EDITION = {Second},
 PUBLISHER = {Cambridge University Press, Cambridge},
      YEAR = {2013},
     PAGES = {xviii+643},
      ISBN = {978-0-521-54823-6},
   MRCLASS = {15-01},
  MRNUMBER = {2978290},
MRREVIEWER = {Mohammad Sal Moslehian},
}

@article{Cao2018GossipBasedLB,
  title={Gossip-Based Load Balance Strategy in Big Data Systems with Hierarchical Processors},
  author={Xiufeng Cao and Shu Gao and Liangchen Chen},
  journal={Wireless Personal Communications},
  year={2018},
  volume={98},
  pages={157-172}
}

@inproceedings{pmlr-v97-koloskova19a,
  title={Decentralized stochastic optimization and gossip algorithms with compressed communication},
  author={Koloskova, Anastasia and Stich, Sebastian and Jaggi, Martin},
  booktitle={International Conference on Machine Learning},
  pages={3478--3487},
  year={2019},
  organization={PMLR}
}

@INPROCEEDINGS{6855039,
  author={Lazzeretti, Riccardo and Horn, Steven and Braca, Paolo and Willett, Peter},
  booktitle={2014 IEEE International Conference on Acoustics, Speech and Signal Processing (ICASSP)},
  title={Secure multi-party consensus gossip algorithms},
  year={2014},
  volume={},
  number={},
  pages={7406-7410},
  doi={10.1109/ICASSP.2014.6855039}}

@article {Hunter2010,
    AUTHOR = {Hunter, Jeffrey J.},
     TITLE = {Some stochastic properties of ``semi-magic'' and ``magic''
              {M}arkov chains},
   JOURNAL = {Linear Algebra Appl.},
  FJOURNAL = {Linear Algebra and its Applications},
    VOLUME = {433},
      YEAR = {2010},
    NUMBER = {5},
     PAGES = {893--907},
      ISSN = {0024-3795},
   MRCLASS = {60J10 (60J20)},
  MRNUMBER = {2658640},
MRREVIEWER = {Ross P. Kindermann},
       DOI = {10.1016/j.laa.2010.04.021},
       URL = {https://doi-org.acces.bibl.ulaval.ca/10.1016/j.laa.2010.04.021},
}

@article {Perron1907,
	AUTHOR     = {Perron, Oskar},
	TITLE      = {Zur {T}heorie der {M}atrices},
	JOURNAL    = {Math. Ann.},
	FJOURNAL   = {Mathematische Annalen},
	VOLUME     = {64},
	YEAR       = {1907},
	NUMBER     = {2},
	PAGES      = {248--263},
	ISSN       = {0025-5831},
	MRCLASS    = {DML},
	MRNUMBER   = {1511438},
	DOI        = {10.1007/BF01449896},
	URL        = {https://doi-org.acces.bibl.ulaval.ca/10.1007/BF01449896},
}

@book{Frobenius1912,
	title      = {{\"U}ber Matrizen aus nicht negativen Elementen},
	author     = {Frobenius, Ferdinand Georg},
	isbn       = {9783111271903},
	series     = {Preussische Akademie der Wissenschaften Berlin: Sitzungsberichte der Preu{\ss}ischen Akademie der Wissenschaften zu Berlin},
	url        = {https://books.google.ca/books?id=fuK3PgAACAAJ},
	year       = {1912},
	publisher  = {Reichsdr.}
}

@article {Mirsky1965,
	AUTHOR = {Perfect, Hazel and Mirsky, Leon},
	TITLE = {Spectral properties of doubly-stochastic matrices},
	JOURNAL = {Monatsh. Math.},
	FJOURNAL = {Monatshefte f\"{u}r Mathematik},
	VOLUME = {69},
	YEAR = {1965},
	PAGES = {35--57},
	ISSN = {0026-9255},
	MRCLASS = {15.65},
	MRNUMBER = {175917},
	MRREVIEWER = {J. G. Mauldon},
	DOI = {10.1007/BF01313442},
	URL = {https://doi-org.acces.bibl.ulaval.ca/10.1007/BF01313442},
}

@article {Marcus1961,
	AUTHOR = {Marcus, Marvin and Minc, Henryk and Moyls, Benjamin},
	TITLE = {Some results on non-negative matrices},
	JOURNAL = {J. Res. Nat. Bur. Standards Sect. B},
	FJOURNAL = {Journal of Research of the National Bureau of Standards.
	Section B. Mathematics and Mathematical Physics},
	VOLUME = {65B},
	YEAR = {1961},
	PAGES = {205--209},
	ISSN = {0022-4340},
	MRCLASS = {15.60},
	MRNUMBER = {125124},
	MRREVIEWER = {L. Mirsky},
}

@book {Levin2009,
    AUTHOR = {Levin, David A. and Peres, Yuval and Wilmer, Elizabeth L.},
     TITLE = {Markov chains and mixing times},
      NOTE = {With a chapter by James G. Propp and David B. Wilson},
 PUBLISHER = {American Mathematical Society, Providence, RI},
      YEAR = {2009},
     PAGES = {xviii+371},
      ISBN = {978-0-8218-4739-8},
   MRCLASS = {60J10 (60-01 60J05 60K35 60K37 68U20 68W20)},
  MRNUMBER = {2466937},
MRREVIEWER = {Olle H\"{a}ggstr\"{o}m},
       DOI = {10.1090/mbk/058},
       URL = {https://doi-org.acces.bibl.ulaval.ca/10.1090/mbk/058},
}

@article {Sinkhorn,
    AUTHOR = {Sinkhorn, Richard and Knopp, Paul},
     TITLE = {Concerning nonnegative matrices and doubly stochastic
              matrices},
   JOURNAL = {Pacific J. Math.},
  FJOURNAL = {Pacific Journal of Mathematics},
    VOLUME = {21},
      YEAR = {1967},
     PAGES = {343--348},
      ISSN = {0030-8730},
   MRCLASS = {15.65},
  MRNUMBER = {210731},
MRREVIEWER = {J. G. Mauldon},
       URL = {http://projecteuclid.org.acces.bibl.ulaval.ca/euclid.pjm/1102992505},
}

@article {Maksimov1965,
    AUTHOR = {Maksimov, Valeri M.},
     TITLE = {On the relation between limit theorems on finite groups and an
              ergodic theorem for {M}arkov chains with doubly-stochastic
              transition matrix},
   JOURNAL = {Teor. Verojatnost. i Primenen},
  FJOURNAL = {Akademija Nauk SSSR. Teorija Verojatnoste\u{\i} i ee Primenenija},
    VOLUME = {10},
      YEAR = {1965},
     PAGES = {544--547},
      ISSN = {0040-361x},
   MRCLASS = {60.65 (60.08)},
  MRNUMBER = {0199893},
MRREVIEWER = {L. Schmetterer},
}

@article {Maksimov1970,
    AUTHOR = {Maksimov, Valeri M.},
     TITLE = {The convergence of nonhomogeneous doubly stochastic {M}arkov
              chains},
   JOURNAL = {Teor. Verojatnost. i Primenen.},
  FJOURNAL = {Akademija Nauk SSSR. Teorija Verojatnoste\u{\i} i ee Primenenija},
    VOLUME = {15},
      YEAR = {1970},
     PAGES = {622--636},
      ISSN = {0040-361x},
   MRCLASS = {60J10},
  MRNUMBER = {0298773},
MRREVIEWER = {G. E. Denzel},
}

@article {Du1987,
    AUTHOR = {Du, Hong Ke},
     TITLE = {The limit of powers of doubly stochastic matrices},
   JOURNAL = {Math. Practice Theory},
  FJOURNAL = {Mathematics in Practice and Theory. Shuxue de Shijian yu
              Renshi},
      YEAR = {1987},
    NUMBER = {3},
     PAGES = {76--77},
      ISSN = {1000-0984},
   MRCLASS = {15A51},
  MRNUMBER = {919339},
MRREVIEWER = {Shao Kuan Li},
}

@article {Hwang2001,
    AUTHOR = {Hwang, Suk-Geun and Pyo, Sung-Soo},
     TITLE = {Doubly stochastic matrices whose powers eventually stop},
   JOURNAL = {Linear Algebra Appl.},
  FJOURNAL = {Linear Algebra and its Applications},
    VOLUME = {330},
      YEAR = {2001},
    NUMBER = {1-3},
     PAGES = {25--30},
      ISSN = {0024-3795},
   MRCLASS = {15A51},
  MRNUMBER = {1826646},
       DOI = {10.1016/S0024-3795(01)00260-9},
       URL = {https://doi-org.acces.bibl.ulaval.ca/10.1016/S0024-3795(01)00260-9},
}

@article {Renier2009,
    AUTHOR = {R\'{e}nier, Simon},
     TITLE = {Characterization, parity, and power sequences of locally
              finite doubly stochastic matrices},
   JOURNAL = {Discrete Math.},
  FJOURNAL = {Discrete Mathematics},
    VOLUME = {309},
      YEAR = {2009},
    NUMBER = {23-24},
     PAGES = {6563--6571},
      ISSN = {0012-365X},
   MRCLASS = {15B51},
  MRNUMBER = {2558621},
MRREVIEWER = {Kent E. Morrison},
       DOI = {10.1016/j.disc.2009.07.005},
       URL = {https://doi-org.acces.bibl.ulaval.ca/10.1016/j.disc.2009.07.005},
}

@book {Fritz1979,
    AUTHOR = {Fritz, Franz-Josef and Huppert, Bertram and Willems, Wolfgang},
     TITLE = {Stochastische {M}atrizen},
    SERIES = {Hochschultext [University Textbooks]},
 PUBLISHER = {Springer-Verlag, Berlin-New York},
      YEAR = {1979},
     PAGES = {vii+192},
      ISBN = {3-540-09126-2},
   MRCLASS = {60J10 (15A51 60J20)},
  MRNUMBER = {519871},
MRREVIEWER = {E. Seneta},
}

@article {Hensel1926,
    AUTHOR = {Hensel, Kurt},
     TITLE = {\"{U}ber {P}otenzreihen von {M}atrizen},
   JOURNAL = {J. Reine Angew. Math.},
  FJOURNAL = {Journal f\"{u}r die Reine und Angewandte Mathematik. [Crelle's
              Journal]},
    VOLUME = {155},
      YEAR = {1926},
     PAGES = {107--110},
      ISSN = {0075-4102},
   MRCLASS = {DML},
  MRNUMBER = {1581075},
       DOI = {10.1515/crll.1926.155.107},
       URL = {https://doi-org.acces.bibl.ulaval.ca/10.1515/crll.1926.155.107},
}

\end{document}